\documentclass[leqno,11pt]{amsart}
\usepackage{amssymb, amsmath}
\def\refer#1{~\ref{#1}}
\def\refeq#1{~(\ref{#1})}
\def\ccite#1{~\cite{#1}}

\def\suite#1#2#3{(#1_{#2})_{#2\in {#3}}}
\def\longformule#1#2{
\displaylines{ \qquad{#1} \hfill\cr \hfill {#2} \qquad\cr } }
\def\inte#1{
\displaystyle\mathop{#1\kern0pt}^\circ }

\def\sumetage#1#2{\sum_{\substack{{#1}\\{#2}}}}

\let\al=\alpha

\let\g=\gamma
\let\d=\delta
\let\e=\varepsilon
\let\ep=\varepsilon

\let\lam=\lambda

\let\s=\sigma

\let\D=\Delta

\let\wt=\widetilde
\let\wh=\widehat

\let\ov\overline

\def\cB{{\mathcal B}}
\def\cC{{\mathcal C}}

\def\cF{{\mathcal F}}
\def\cG{{\mathcal G}}
\def\cH{{\mathcal H}}
\def\cI{{\mathcal I}}
\def\cJ{{\mathcal J}}
\def\cK{{\mathcal K}}
\def\cL{{\mathcal L}}

\def\cO{{\mathcal O}}

\def\cR{{\mathcal R}}
\def\cS{{\mathcal S}}
\def\cT{{\mathcal T}}
\def\cU{{\mathcal U}}
\def\cV{{\mathcal V}}
\def\cW{{\mathcal W}}
\def\cX{{\mathcal X}}

\def\virgp{\raise 2pt\hbox{,}}
\def\cdotpv{\raise 2pt\hbox{;}}
\def\eqdefa{\buildrel\hbox{\footnotesize def}\over =}
\def\Id{\mathop{\rm Id}\nolimits}

\def\sgn{\mathop{\rm sgn}\nolimits}

\def\C{\mathop{\mathbb C\kern 0pt}\nolimits}
\def\DD{\mathop{\mathbb D\kern 0pt}\nolimits}
\def\EE{\mathop{{\mathbb E \kern 0pt}}\nolimits}
\def\K{\mathop{\mathbb K\kern 0pt}\nolimits}
\def\N{\mathop{\mathbb N\kern 0pt}\nolimits}
\def\Q{\mathop{\mathbb Q\kern 0pt}\nolimits}
\def\R{{\mathop{\mathbb R\kern 0pt}\nolimits}}
\def\SS{\mathop{\mathbb S\kern 0pt}\nolimits}
\def\ZZ{\mathop{\mathbb Z\kern 0pt}\nolimits}
\def\TT{\mathop{\mathbb T\kern 0pt}\nolimits}
\def\P{\mathop{\mathbb P\kern 0pt}\nolimits}
\def \H{{\mathop {\mathbb H\kern 0pt}\nolimits}}
\newcommand{\ds}{\displaystyle}

\newcommand{\Z}{{\ZZ}}

\newcommand{\beq}{\begin{equation}}
\newcommand{\eeq}{\end{equation}}
\newcommand{\ben}{\begin{eqnarray}}
\newcommand{\een}{\end{eqnarray}}
\newcommand{\beno}{\begin{eqnarray*}}
\newcommand{\eeno}{\end{eqnarray*}}
\newcommand{\bqs}{\begin{equation*}}
\newcommand{\eqs}{\end{equation*}}
\newcommand{\andf}{\quad\hbox{and}\quad}
\newcommand{\with}{\quad\hbox{with}\quad}

\def \cFH {\cF_\H}

\def\equivH#1 {\buildrel\hbox{\tiny {$#1$}}\over \equiv}
\def\simH#1 {\buildrel\hbox{\footnotesize {$#1$}}\over \sim}

\newtheorem{definition}{Definition}[section]
\newtheorem{theorem}{Theorem}[section]
\newtheorem{lemma}{Lemma}[section]
\newtheorem{remark}{Remark}[section]
\newtheorem{cor}{Corollary}[section]
\newtheorem{proposition}{Proposition}[section]
\numberwithin{equation}{section}

\begin{document}
\title[A frequency  space  for the Heisenberg group]
{A frequency space  for the Heisenberg group}
 \author[H. BAHOURI] {Hajer Bahouri} 
 \address[H. Bahouri]%
{LAMA, UMR 8050\\
Universit\'e Paris-Est Cr\'eteil, 94010 Cr\'eteil Cedex, FRANCE}
  \email{hajer.bahouri@math.cnrs.fr}
  \email{}    
 \author[J.-Y. CHEMIN]{Jean-Yves Chemin}
\address [J.-Y. Chemin]%
{Laboratoire J.-L. Lions, UMR 7598 \\
Universit\'e Pierre et Marie Curie, 75230 Paris Cedex 05, FRANCE }
\email{chemin@ann.jussieu.fr}
\author[R. DANCHIN]{Raphael Danchin }%
\address[R. Danchin] {LAMA, UMR 8050\\
Universit\'e Paris-Est Cr\'eteil, 94010 Cr\'eteil Cedex, FRANCE}
  \email{danchin@u-pec.fr}
 
\date{\today}

\begin{abstract} 
We here revisit  Fourier analysis on the Heisenberg group $\H^d.$
Whereas, according to  the standard definition, 
 the Fourier transform of an integrable function $f$ on $\H^d$
is  a one parameter family of bounded operators on $L^2(\R^d),$
we define (by taking  advantage of basic properties of Hermite functions)
the Fourier transform $\wh f_\H$ of $f$ to be a  uniformly continuous mapping  
 on the  set $\wt\H^d\eqdefa\N^d\times\N^d\times\R\setminus\{0\}$
endowed with a suitable distance~$\wh d.$
This enables us to extend  $\wh f_\H$ to the completion $\wh\H^d$ of $\wt\H^d,$ 
and to get  an \emph{explicit} 
asymptotic description of  the Fourier transform when the `vertical' frequency tends to $0.$
 
We expect our approach to be relevant for adapting to the Heisenberg framework
a number of classical results for the $\R^n$ case that are based on Fourier analysis. 
As an example, we  here  establish an explicit extension of    
the Fourier transform for smooth functions on $\H^d$ 
that are independent of the vertical variable. 
\end{abstract}

\maketitle


\vskip 0.2cm


\setcounter{equation}{0}
\section*{Introduction}\label {intro}

Fourier analysis on locally compact Abelian groups is by now classical matter 
that goes back to the first half of the 20th century (see e.g. \cite{rudin} for 
a self-contained presentation). 

Consider a locally compact Abelian group $(G,+)$ endowed with a Haar measure $\mu,$
and denote by  $(\wh G,\cdot)$ the dual group of $(G,+)$
  that is the set of characters on $G$ endowed with the standard multiplication of functions. 
By definition,  the Fourier transform of an integrable function $f:G\to\C$ is the 
continuous and bounded function $\wh f:\wh G\to\C$  (also denoted by $\cF f$)
defined by  \begin{equation}\label{def:FG}
\forall \g\in\wh G,\; \wh f(\g) = \cF f(\g)\eqdefa\int_G  f(x)\,\overline{\gamma(x)}\,d\mu(x).
\end{equation}
Being also a locally compact Abelian group, the `frequency space' $\wh G$ may be  endowed with 
a Haar measure $\wh\mu.$  It turns out  to be possible to normalize $\wh\mu$  so 
that the following \emph{Fourier inversion formula} holds true
for, say, all  function $f$ in $L^1(G)$ with $\wh f$ in $L^1(\wh G)$:
\begin{equation}\label{eq:inversionG}
\forall x\in G,\; f(x)=\int_{\wh G} \wh f(\gamma)\,\gamma(x)\,d\wh\mu(\gamma).
\end{equation}
As a consequence, we get the Fourier-Plancherel identity
\begin{equation}\label{eq:FPG}
\int_G|f(x)|^2\,d\mu(x)=\int_{\wh G} |\wh f(\g)|^2\,d\wh\mu(\g)
\end{equation}
for all $f$ in $L^1(G)\cap L^2(G).$
\medbreak
Fourier transform on locally compact Abelian groups has a number of other interesting 
properties that we do not wish to enumerate here. Let us just recall
that it changes convolution products into products of functions, namely
\begin{equation}\label{eq:convG}
\forall f\in L^1(G),\;\forall g\in L^1(G),\;\cF(f\star g)=\cF f\cdot\cF g.
\end{equation}

In the   Euclidean case of $\R^n$  the  dual group may be identified to~$(\R^n)^\star$ 
 through the map~$\xi \mapsto  e^{i\langle \xi ,\cdot\rangle}$ (where $\langle \cdot,\cdot\rangle$ 
 stands for the duality bracket  between~$(\R^n)^\star $ and~$\R^n$),  and the Fourier transform
 of an integrable function~$f$  may  thus be seen as the function on $(\R^n)^\star$ (usually identified
 to $\R^n$)  given by
\beq
\label {definFourierclassic}
\cF (f) (\xi) = \wh f(\xi)\eqdefa \int_{\R^n} e^{-i\langle  \xi,x\rangle } f(x)\, dx.
\eeq

Of course,  we have \eqref{eq:convG}  and, as is well known,  if one endows the frequency  space $(\R^n)^\star$ with the measure $\frac1{(2\pi)^n}d\xi$ then the inversion and Fourier-Plancherel 
formulae \eqref{eq:inversionG} and \eqref{eq:FPG} hold true.
Among the numerous additional properties of the Fourier transform on $\R^n,$ let
us just underline that it allows to  `diagonalize'    the Laplace operator, namely for all
 smooth compactly supported  functions, we have
\beq
\label {diagDeltaRd}
\cF(\Delta f) (\xi) =- |\xi|^2 \wh f(\xi).
\eeq

For noncommutative groups,  Fourier theory  gets wilder, 
for  the dual group is too `small' to keep the definition of the Fourier transform given in
 \eqref{def:FG} and still have the inversion formula\refeq{eq:inversionG}. 
 Nevertheless,  if the group  has `nice' properties
 (that we wish not to list here) then one can  work out a consistent Fourier theory 
 with properties analogous to \eqref{eq:inversionG}, \eqref{eq:FPG} and \eqref{eq:convG}
 (see e.g.\ccite{astengo2,crs,corwingreenleaf,hula,RS,stein2,taylor1,thangavelu} and the references therein for the case of nilpotent Lie groups).
 In that context,  the classical definition of the Fourier transform 
 amounts to replacing   characters in \eqref{def:FG} with suitable families of irreducible representations
 that are valued in Hilbert spaces (see e.g. \cite{corwingreenleaf,folland} for a detailed presentation). 
Consequently, the Fourier transform is no longer a complex valued function
but rather a family of bounded operators on suitable Hilbert spaces. 
It goes without saying that within this approach, the notion of `frequency space' becomes unclear, 
which makes Fourier theory much more cumbersome than in the Abelian case. 

In the present paper, we want to focus on the Heisenberg group which, to some 
extent, is the simplest noncommutative nilpotent Lie group
and comes into play in various areas of mathematics,
ranging  from complex analysis to geometry or number  theory,   probability  theory, quantum mechanics and partial differential equations (see e.g. \cite{bfg, farautharzallah, stein2,taylor1}). 
As  several  equivalent definitions  coexist  in the literature, 
 let us  specify the one that  we shall adopt throughout. 
\begin{definition}
{\sl 
Let~$\s(Y,Y') =\langle \eta,y'\rangle -\langle \eta',y\rangle$ be  the canonical symplectic form 
on~$T^\star\R^d.$
The Heisenberg group~$\H^d$  is the  set~$T^\star\R^d \times\R$  equipped with the product law
  $$
w\cdot w'\eqdefa
\bigl(Y+Y' , s+s'+ 2\s(Y,Y')\bigr) = \bigl(y+y',  \eta+\eta' , s+s'+2 \langle \eta,y'\rangle -2\langle \eta',y\rangle\bigr)
$$
where~$w=(Y,s)=(y,\eta,s)$ and $w'=(Y',s')=(y',\eta',s')$ are  generic elements of~$\H^d.$}
\end{definition}
As regards topology and measure theory on the Heisenberg group, 
we shall look at~$\H^d$ as the set $\R^{2d+1},$ after identifying 
  $(Y,s)$ in $\H^d$ to $(y,\eta,s)$ in $\R^{2d+1}.$ 
 With this viewpoint, the  \emph{Haar measure} on $\H^d$ is just  the  Lebesgue measure on $\R^{2d+1}.$
  In particular, one can define the following convolution product  for any two integrable functions~$f$ and~$g$:
\beq
\label {definConvolH}
f \star g ( w ) \eqdefa \int_{\H^d} f ( w \cdot v^{-1} ) g( v)\, dv 
= \int_{\H^d} f ( v ) g( v^{-1} \cdot w)\, dv.
\eeq
Even though
convolution on the Heisenberg group is noncommutative, if one
 defines the \emph{Lebesgue spaces} $L^p(\H^d)$  to be just  $L^p(\R^{2d+1}),$ then
one  still gets the classical Young inequalities in that context. 
  \smallbreak
  As already explained above, as  $\H^d$ is noncommutative, 
 in order to have a good Fourier theory, one has to 
 resort to more elaborate irreducible representations
 than character. In fact, the group of characters on~$\H^d$ is isometric to the group  of characters  on~$T^\star\R^d.$ Hence,  if one defines the Fourier transform 
 according to  \eqref{def:FG}  then  the information pertaining to the vertical variable~$s$ is  lost. 
 
  There are essentially two (equivalent) approaches. They  are based either on the  \emph{Bargmann representation} or on the \emph{Schr\"odinger representation} (see \cite{corwingreenleaf}).  For simplicity, let us just recall the second one 
 which is  the  family of  
group homomorphisms $w\mapsto U^\lam_w$ (with $\lambda\in\R\setminus\{0\}$)
between~$\H^d$ and the  unitary group~$\cU(L^2(\R^d))$ of~$L^2(\R^d),$
defined for all~$w=(y,\eta, s)$ in~$\H^d$ and $u$ in $L^2(\R^d)$ by 
$$
U^\lam _w u(x)\eqdefa e^{-i\lam (s+2\langle \eta, x-y\rangle)} u(x-2y).
$$ 
The classical definition of Fourier transform of integrable functions on $\H^d$ reads  as follows:
\begin{definition}
\label {definFourierSchrodinger}
{\sl   The\emph{ Fourier transform}   \index{Fourier!transform}
of  an integrable function~$f$ on~$\H^d$ is the family $(\cF^\H(f)(\lambda))_{\lambda\in\R\setminus\{0\}}$ of bounded operators on $L^2(\R^d)$ given  by 
$$
\cF^{\H} (f)(\lam) \eqdefa \int_{\H^d} f(w) U^\lam_w\, dw.
$$
}
\end{definition}

In the present paper, we strive for  another  definition of the Fourier 
transform, that is as similar as possible 
to the one for locally compact groups given in \eqref{def:FG}. 
 In particular, we want the Fourier transform 
to be a complex valued function defined on some explicit `frequency space' that  may be endowed
with a structure of a locally compact and complete metric space, 
and to get formulae similar to \eqref{eq:inversionG}, \eqref{eq:FPG}, \eqref{eq:convG}
together with a diagonalization of the Laplace operator (for the Heisenberg group of course)
analogous to \eqref{diagDeltaRd}.

There is a number of motivations for our approach. An important one is that, having an explicit
frequency space will  allow us to get elementary  proofs of the basic results involving the Fourier transform, 
just by mimicking the corresponding ones of the Euclidean setting. 
In particular, we expect our setting to  open the way  to new results  for 
partial differential equations on the Heisenberg group. 
Furthermore, our definition will enable us to get
an explicit (and comprehensible) description of the range of the Schwartz space by the Fourier transform.
As a consequence, extending the Fourier transform  to the set of tempered distributions
will become rather elementary
(see more details in our forthcoming paper \cite{bcdh}).

 In the present paper, we will give two concrete applications
 of our approach. First, in Theorem \ref{FourierL1basic}, we will  provide an explicit 
asymptotic description of the Fourier transform when (what plays the role of) 
the vertical frequency parameter  tends to $0.$  
Our second application is the extension (also explicit) 
of the Fourier transform to functions 
depending only on the horizontal variable (this is Theorem \ref{Fourierhorizontal}).

 \bigbreak\noindent{\bf Acknowledgments:}
  The authors wish to thank very warmly  Fulvio  Ricci for enlightening
 suggestions that played a decisive role in the
construction of this text. They are also indebted to
 Nicolas Lerner for  enriching  discussions.

 A determining part of this work has been
carried out in  the exceptional environment of  the \emph{Centre
International  de Rencontres Math\'ematiques} in Luminy.

The third author has been partially supported by the \emph{Institut
Universitaire de France}.


\section{Results}\label{s:results}

  Before presenting  the main results of the paper, let us recall 
  how, with the standard definition of the Fourier transform in $\H^d,$
  Properties \eqref{eq:inversionG}, \eqref{eq:FPG} and \eqref{eq:convG}
  may be stated (the reader may refer   to e.g.\ccite{bfg, Beals, farautharzallah, fisher, folland, geller2, hula, thangavelu2, stein2, taylor1, thangavelu} for more details).  
\begin{theorem}
\label {RecallClassicla FourierH}
{\sl Let~$f$ be an integrable function. Then  we have
\beq
\label {L1LinftyFourierbasic}
\forall\lam \in \R\setminus\{0\}\,,\  \|\cF^{\H}(f)(\lam)\|_{\cL(L^2)} \leq \|f\|_{L^1(\H^d)}
\eeq
and,  for any function~$u$ in~$L^2(\R^d)$,  the map $\lam \mapsto  \cF^\H(f) (\lam)(u)$
is continuous from $\R\setminus\{0\}$ to~$L^2(\R^d).$
\medbreak
 For any function~$f$ in the Schwartz space $\cS(\H^d)$ (which is  the classical Schwartz space on~$\R^{2d+1}$),   we have the inversion formula: 
\beq
\label {inversionHclassical}
 \forall w\in\H^d,\; f(w) =   \frac {2^{d-1}}  {\pi^{d+1} }   \int_{\R} {\rm tr}\bigl( U^\lam_{w^{-1}}\cF^\H f(\lam)\bigr)  |\lam|^d d\lam\,,
\eeq
where~${\rm tr}(A)$ denotes the trace of the operator~$A$.
\medbreak
 Moreover, if~$f$ belongs to~$L^1(\H^d)\cap L^2(\H^d)$ then for any~$\lam$ in~$\R\setminus\{0\}$,~$\cF^\H(f)(\lam)$ is an Hilbert-Schmidt operator, and we have
\beq 
\label {FourierPlancherelHclassical}
\|f\|_{L^2(\H^d)}^2= \frac  {2^{d-1}} {\pi^{d+1}} \int_{\R}\| \cF^\H(f)(\lam)\|_{HS} ^2 \,|\lam|^dd\lam
\eeq 
where~$\|\cdot\|_{HS}$ stands for the Hilbert-Schmidt norm.
}
\end{theorem}

We also have an analogue of the convolution identity \eqref{eq:convG}. 
Indeed, as the map~$w\mapsto U^\lam_w$ is a homomorphism
between~$\H^d$ and~$\cU(L^2(\R^d))$, we
get for any integrable functions $f$ and~$g,$
\beq
\label {FourierConvol} 
\cF^{\H} (f\star g) (\lam)  =  \cF^{\H} (f)(\lam)\circ \cF^{\H}(g)(\lam).
\eeq

Let us next recall the definition of the (sub-elliptic) Laplacian on the Heisenberg group, that
will play a fundamental role in our approach.
Being  a real Lie group,  the Heisenberg group
  may be equipped with a linear space   of   \emph{left invariant} vector fields,
that is vector fields commuting with any   left translation~$\tau_w(w') \eqdefa w\cdot w'$.
It is well known that this linear space has dimension $2d+1$ and  is generated by the vector fields
$$ S\eqdefa\partial_s\,,\ \  \cX_j\eqdefa\partial_{y_j} +2\eta_j\partial_s\andf \Xi_j\eqdefa \partial_{\eta_j} -2y_j\partial_s\,,\ 1\leq j\leq d.
$$
The  \emph{Laplacian} \index{Laplacian}
associated to the vector fields~$(\cX_j)_{1\leq j\leq d}$ and~$(\Xi_j)_{1\leq j\leq d}$ reads
\begin{equation}
\label{defLaplace}
\D_{\H}  \eqdefa \sum_{j=1} ^d (\cX_j^2+\Xi_j^2). 
\end{equation}

As in the Euclidean case (see Identity\refeq{diagDeltaRd}), 
Fourier transform allows to diagonalize Operator $\D_\H$: it  is  based on the following relation 
that holds true for all functions $f$ and $u$ in $\cS(\H^d)$ and $\cS(\R^d),$ respectively (see e.g.\ccite{huet, O}):
\beq
\label {FourierEtLaplace}
\cF^{\H}(\D_\H f) (\lam) = 4\cF^{\H}(f)(\lam)  \circ  \D_{\rm osc} ^\lam \with \D_{\rm osc}^\lam u (x) \eqdefa\sum_{j=1}^d \partial_j^2 u(x) - \lam^2|x|^2 u(x).
\eeq
This prompts us  to take advantage of  the spectral structure of the harmonic oscillator to get an analog of Formula\refeq {diagDeltaRd}.
To this end, we need to introduce the family of Hermite functions~$(H_n)_{n\in\N^d}$
defined by 
\beq
\label{Hermite functions}
H_n \eqdefa  \Bigl(\frac 1 {2^{|n|} n!}\Bigr) ^{\frac 12}C^n H_0 \with C^n \eqdefa \prod_{j=1}^d C_j^{n_j}
\andf H_0(x)\eqdefa \pi^{-\frac d 2} e^{-\frac {|x|^2} 2},
\eeq
where~$C_j\eqdefa -\partial_j +M_j$ stands for  the 
 \emph{creation operator} with respect to the $j$-th variable
 and~$M_j$ is the multiplication operator  defined by~$M_ju(x)\eqdefa x_ju(x).$
 As usual, $n!\eqdefa n_1!\dotsm n_d!$ and $|n|\eqdefa n_1+\cdots+n_d$. 
\medbreak
It is well known that the family~$ \suite H n {\N^d}$ is an orthonormal basis of~$L^2(\R^d).$  In particular,
\begin{equation}\label{def:kro}
 \forall  (n,m)\in\N^d\times\N^d \,,\ (H_n|H_m)_{L^2}=\delta_{n,m} ,
\end{equation}
where $\delta_{n,m}=1$ if $n=m,$ and $\delta_{n,m}=0$ if $n\not=m.$
\medbreak
Besides,  we have
\beq
\label {relationsHHermite}
( -\partial_j^2+M_j^2) H_n =( 2n_j+1) H_n \quad\hbox{and thus}\quad -\D_{\rm osc}^1 H_n = (2|n|+d) H_n.
\eeq
For~$\lam$ in~$\R\setminus\{0\},$ we further introduce
the rescaled Hermite function~$H_{n,\lam} (x)\eqdefa |\lam|^{\frac d 4} H_n(|\lam|^{\frac 12} x)$. 
It is obvious that~$(H_{n,\lam})_{n\in \N^d}$ is still  an orthonormal basis of~$L^2(\R^d)$ and that 
\beq
\label {relationsHHermiteD}
( -\partial_j^2+\lam^2M_j^2) H_{n,\lam} =( 2n_j+1)|\lam| H_{n,\lam} \quad\hbox{and thus}\quad -\D_{\rm osc}^\lam H_{n,\lam} = (2|n|+d)|\lam| H_{n,\lam}.
\eeq
We are now ready  to give `our' definition of the  Fourier transform on $\H^d.$
\begin{definition}
\label {definFouriercoeffH}
{\sl 
Let~$\wt \H ^d\eqdefa \N^{d}\times\N^d\times \R\setminus\{0\}.$ We denote by~$\wh w=(n,m,\lam)$ a generic point of~$\wt \H^d$. For~$f$ in~$L^1(\H^d)$,  we define 
the map~$\cFH f$ (also denoted by~$\wh f_\H$) to be 
$$
\cFH f: \ 
\left\{
\begin{array}{ccl}
\wt \H ^d  & \longrightarrow & \C\\[1ex]
\wh w & \longmapsto & 
\bigl(\cF^{\H}(f)(\lam) H_{m,\lam} |H_{n,\lam}\bigr)_{L^2}.
\end{array}
\right.
$$
}
\end{definition}

{}From now on, we shall use only  that  definition of the Fourier transform, which  amounts
to  considering the `infinite matrix' of~$\cF^\H f(\lam)$ in the orthonormal basis of~$L^2(\R^d)$ given by~$(H_{n,\lam})_{n\in \N}.$ 
For further purpose, it is in fact much more  convenient to rewrite~$\cF_\H f$  in terms of 
the mean value of~$f$ \emph{modulated by  some oscillatory functions}
which may be seen as  suitable Wigner distribution functions  of  the family~$(H_{n,\lam})_{n\in\N^d,\lam\not=0},$ and  will play the same role as
 the characters~$e^{i\langle \xi, \cdot\rangle}$ in the Euclidean case. 
 Indeed, by definition, we have 
 $$
 \cF_\H f(\wh w)=  \int_{\H^d\times \R^d} f(w) e^{-is\lam} e^{-2i\lam\langle \eta, x-y\rangle} H_{m,\lam} (x-2y) H_{n,\lam} (x) \,dw \,dx.
 $$
 Therefore, making an obvious change of variable, we discover that
\begin{eqnarray}\label {definFourierWigner}
\cF_\H f(\wh w) &&\!\!\!\!\!\!\!= \int_{\H^d}  \overline{e^{is\lam} \cW(\wh w,Y)}\, f(Y,s) \,dY\,ds
\with\\\label{definWigner}
\ds \cW(\wh w,Y) && \!\!\!\!\!\!\!\! \eqdefa \int_{\R^d} e^{2i\lam\langle \eta,z\rangle} H_{n,\lam} (y+z) H_{m,\lam} (-y+z) \,dz.
\end{eqnarray}
At this stage,  looking at the action of the Laplace operator  on functions~$e^{is\lam} \cW(\wh w,Y)$ is illuminating.   
Indeed, easy computations (carried out in Appendix) give
\beq
\label {LaplacianHWignerj}
(\cX_j^2+\Xi_j^2)  \bigl(e^{is\lam} \cW(\wh w,Y)\bigr )= -4|\lam| (2m_j+1) e^{is\lam} \cW(\wh w,Y).
\eeq
By summation on $j\in\{1,\cdots,d\},$ we get
\beq
\label {DeltaWignerHermite}
\Delta_\H \bigl (e^{is\lam} \cW(\wh w,Y) \bigr) = -4|\lam|(2|m|+d) e^{is\lam} \cW(\wh w,Y),
\eeq
from which one may deduce that, whenever $f$ is in $\cS(\H^d)$  (again, refer to the Appendix), 
\beq
\label {FourierdiagDeltaHfond}
\cF_{\H}(\D_\H f) (\wh w) = -4|\lam|(2|m|+d) \wh f_\H(\wh w).
\eeq
Let us underline the similarity with Relation\refeq {diagDeltaRd} pertaining to the Fourier transform in $\R^n.$  

\medbreak

One of the basic principles of the  Fourier transform on $\R^n$ is  that `{\it regularity implies decay}'. It remains
true in the Heisenberg framework, as stated in   the following  lemma.
\begin{lemma}
\label {decaylambdan}
{\sl 
For any non negative integer~$p$, there exist an integer~$N_p$ and a positive constant~$C_p$ such that 
for any~$\wh w$ in~$\wt \H^d$ and any~$f$ in~$\cS(\H^d),$ we have 
\begin{equation}
\label {eq:decay}
\bigl(1+ |\lam|(  |n| + |m|+ d) +|n-m|  \bigr)^p  |\wh f_\H(n,m,\lam)|  \leq  C_p  \|f\|_{N_p,\cS},
\end{equation}
where~$\|\cdot\|_{N,\cS}$ denotes the classical family of semi-norms of~$\cS(\R^{2d+1})$, namely
$$\|f\|_{N,\cS}\eqdefa \sup_{ |\al|\leq N}  \bigl\|(1+|Y|^2+s^2)^{N/2}\,\partial_{Y,s}^\al f \bigr\|_{L^\infty}.$$
}
\end{lemma}

As may be easily checked by the reader, in our setting, there are very simple formulae 
corresponding to \refeq {inversionHclassical} and  \refeq {FourierPlancherelHclassical}, if
the set~$\wt \H^d$ is endowed with the  measure $d\wh w$ defined by: 
\beq
\label {definmeasurewhH}
\int_{\wt \H^d} \theta (\wh w)\, d\wh w\eqdefa \sum_{(n,m)\in \N^{2d}} \int_{\R} \theta (n,m,\lam) |\lam|^d d\lam.
\eeq
Then Theorem\refer{RecallClassicla FourierH} recasts as follows: 
\begin{theorem}
\label {inverseFourier-Plancherel}
{\sl 
Let~$f$ be a function in~$\cS(\H^d)$. Then 
the following inversion formula holds true:\beq
\label {MappingofPHdemoeq1}
f(w) = \frac {2^{d-1}}  {\pi^{d+1} }   \int_{\wt \H^d} 
e^{is\lam} \cW(\wh w, Y)\wh f_\H (\wh w) \, d\wh w.
\eeq
Moreover,  for any function $f$ in~$L^1(\H^d)\cap L^2(\H^d),$  we have
\beq
\label {inverseFouriereq2}
\|\wh f_\H\|_{L^2(\wt \H^d)}^2 = \frac {\pi^{d+1}} {2^{d-1}} \|f\|_{L^2(\H^d)}^2.
\eeq
}
\end{theorem}
In this new setting, the  convolution identity\refeq{FourierConvol}
rewrites  as follows for all integrable functions   $f$ and $g$: 
\begin{multline}
\label {newFourierconvoleq1}
 \cFH (f\star g) (n,m,\lam)   = ( \wh f_\H \cdot \wh g_\H)(n,m,\lam)\\ \with
 ( \wh f_\H  \cdot \wh g_\H)(n,m,\lam)  \eqdefa \sum_{\ell\in \N^{d}} \wh f_\H(n,\ell,\lam)\wh g_\H(\ell,m,\lam).
 \end{multline}
 The reader is referred to the appendix for the proof.

 Next, we aim at endowing the set~$\wt\H^d$ with a structure of metric space. 
 According to the decay inequality\refeq {eq:decay}, it is natural to introduce 
  the following  distance~$\wh d$:
\beq
\label {defindistancewtH}
\wh d(\wh w,\wh w') \eqdefa \bigl|\lam(n+m)-\lam'(n'+m')\bigr|_1 +\bigl |(n-m)-(n'-m')|_1+|\lam-\lam'|,
\eeq
where~$|\cdot|_1$ denotes the~$\ell^1$ norm on~$\R^d$.
\medbreak
At  first glance, the metric space~$(\wt \H^d,\wh d)$  seems to 
be the natural frequency space within our approach. However,  it fails  to be complete, 
which may  be a source of difficulties for further development. 
We thus propose to work with its completion, that   is described in the following proposition.
\begin{proposition}
\label {completionHtilde}
{\sl 
The completion of the set~$\wt \H^d$  for the distance~$\wh d$ is the set~$\wh \H^d$ defined by
$$
\wh \H^d\eqdefa  \bigl(\N^{2d} \times\R\setminus\{0\}\bigr)
 \cup \wh \H^d_0 \with \wh \H^d_0 \eqdefa {\R_{\mp}^d}\times \Z^d
\andf
{\R_{\mp}^d}\eqdefa ((\R_-)^d\cup (\R_+)^d).
$$
On $\wh\H^d,$ the extended distance (still denoted by~$\wh d$) is given by
$$
\begin{aligned}
&\wh d((n,m,\lam),(n',m',\lam')) = \bigl|\lam(n+m)-\lam'(n'+m')\bigr|_1 +\bigl |(m-n)-(m'-n')|_1+|\lam-\lam'|
\\&\hspace{11.5cm} \hbox{if }\ \lam\not=0\ \hbox{ and }\ \lam'\not=0,\\
&\wh d\bigl ((n,m,\lam), (\dot x, k)\bigr ) =  \wh d\bigl ((\dot x, k), (n,m,\lam) \bigr )  \eqdefa |\lam(n+m)-\dot x|_1+ |m-n-k|_1+|\lam|  \ \hbox{ if }\ \lam\not=0, \\
&\wh d\bigl ((\dot x,k), (\dot x', k')\bigr )  = |\dot x-\dot x'|_1+|k-k'|_1.
\end{aligned}$$
}
\end{proposition}

\begin{proof}
Consider a Cauchy sequence~$(n_p,m_p,\lam_p)_{p\in \N}$ in~$(\wt \H^d,\wh d)$. If~$p$ and~$p'$ are large enough, 
then~$|(m_p-n_p)-(m_{p'}-n_{p'})|$ is less than~$1$, and thus~$ m_p-n_p$ has to be a constant, that we denote by~$k$.
Next, we see that~$\suite \lam p \N$ is a Cauchy sequence of real numbers, and thus converges to some $\lambda$ in $\R.$ 
If  $\lambda\not=0$ then our definition of $\wh d$ implies that the sequence~$\suite n p \N$ is constant after a certain index,
and thus converges to some $n$ in~$\N^d.$  Therefore we have~$(n_p,m_p,\lam_p)\to(n,n+k,\lam).$ 

If  $\lam=0$ then the Cauchy sequence~$\bigl(\lam_p(n_p+m_p)\bigr)_{p\in \N}$ has to converge to some $\dot x$ in $\R^d.$
By definition of the extended distance, 
it is clear that~$(n_p,m_p, \lam_p)_{p\in \N}$ converges to~$(\dot x,k)$ in~$\wh \H^d$.
Now, if~$\dot x\not=0$  then there exists some  index~$j$ such that~$\dot x_j\not=0$.
Because the sequence~$\bigl (\lam_p(n_{j,p}+m_{j,p})\bigr)_{p\in \N}$ tends to~$\dot x_j$
and~$n_{j,p}+m_{j,p}$ is positive (for large enough $p$), we must have~$\sgn(\lam_p)=\sgn(\dot x_j).$ 
Therefore, all the components of~$\dot x$ have the same sign.
\smallbreak
Conversely, let us prove that any point of~$\R_+^d\times \Z^d$  (the case of~$\R_-^d\times \Z^d$ being similar) is the limit in the sense of~$\wh d$ of 
some   sequence~$(n_p,m_p,\lam_p)_{p\in \N}. $ As~$\Q$ is dense in~$\R$,   there exist two families of sequences of positive integers~$(a_{j,p} )_{p\in \N}$ and~$(b_{j,p} )_{p\in \N}$  such that 
$$
\forall j \in \{1,\cdots,d\}\,,\ \dot x_j  = \lim_{p\rightarrow \infty}  \dot x_{j,p} \with \dot x_{j,p} \eqdefa \frac {a_{j,p} } {b_{j,p} } \andf \lim_{p\rightarrow \infty} b_{j,p} =+\infty.
$$
Let us write that
$$
 \dot x_{p}  = 2\lam_p n_p \with \lam_p \eqdefa \Bigl(2 \prod_{j=1}^d b_{j,p}\Bigr)^{-1} \,,\  n_p \eqdefa
 \Bigl(a_{j,p}  \prod_{j'\not = j}^d b_{j,p}\Bigr)_{1\leq j\leq d} \andf m_p\eqdefa n_p+k.
$$
As~$\suite \lam p \N$ tends to~$0$, we have that~$\ds \lim_{p\rightarrow\infty}
\wh d\bigl ((n_p, n_p+k,\lam_p), (\dot x,k)\bigr) $ converges to~$0$.
\end{proof}

\begin{remark}
{\sl It is not difficult to check that  the closed bounded subsets  of~$\wh \H^d$ (for the distance~$\wh d$) are compact. The details are left to the reader.}
\end{remark}

The above proposition prompts us  to extend the Fourier transform of an integrable function, to the 
frequency set~$\wh\H^d,$ that will play the same role as~$(\R^n)^\star$ in the case of~$\R^n$. With this new point of view, we expect the Fourier transform of any integrable function
to be continuous on the whole $\wh\H^d.$ This is exactly what is stated  in the  following theorem. 
\begin{theorem}
\label {FourierL1basic}
{\sl
The Fourier transform $\wh f_\H$ of any integrable function on $\H^d$ may be extended continuously 
to the whole set $\wh\H^d.$ Still denoting by  $\wh f_\H$ (or $\cF_\H f$) that extension, 
the  linear map~$\cF_\H: f\mapsto \wh f_\H$ is continuous from  the space $L^1(\H^d)$  to the space 
$\cC_0(\wh\H^d)$  of continuous functions  on~$\wh \H^d$ tending to~$0$ at infinity. 
Moreover, we have for all~$(\dot x,k)$ in~$\wh \H^d_0$,
\ben
\label {FourierL1basiceq2}
 \cF_\H f(\dot x,k)  &=  & \int_{T^\star \R^d}  \ov\cK_d(  \dot x ,k,Y) f(Y,s)\, dYds \with\\
\nonumber \cK_d (\dot x, k,Y) & = & \bigotimes_{j=1}^d \cK (\dot x_j, k_j,Y_j)\andf \\
\label {FourierL1basiceq2b}\cK(\dot x, k,y,\eta) &\eqdefa &  
\frac1{2\pi} \!\int_{-\pi}^{\pi}  e^{i \left(2|\dot x|^{\frac 12} (y\sin z + \eta \sgn(\dot x) \cos z) +kz\right)}\, dz.
\een
}
\end{theorem}
In other words, for any sequence~$(n_p, \lam_p)_{p\in \N}$ of~$\N^d\times (\R\setminus\{0\})$ such that
$$
\lim_{p\rightarrow \infty } \lam_p=0\andf \lim_{p\rightarrow \infty } {\lam_p n_p} =\frac {\dot x} 2 \,\virgp
$$
we have
$$
\lim_{p\rightarrow \infty } \wh f_\H (n_p,n_p+k,\lam_p) = \int_{T^\star \R^d} \ov\cK_d(  \dot x ,k,Y) f(Y,s)\, dYds .
$$

Granted with  the above result,  one can propose
a natural extension of the Fourier transform to (smooth) functions on $\H^d$ 
\emph{independent of the vertical variable~$s.$}
This will come up as  a consequence of the following theorem.
\begin{theorem}
\label {Fourierhorizontal}
{\sl  Let us define the following operator $\cG_\H$ on~$L^1(T^\star\R^d)$:
$$
\cG_\H g (\dot x,k) \eqdefa  \int_{T^\star \R^d}  \ov\cK_d(  \dot x ,k,Y) g(Y)\, dY.
$$
Then, for any function $\chi$  in~$\cS(\R)$ with value~$1$ at~$0$ and  compactly supported Fourier transform, and  any  function~$g$ in~$L^1(T^\star \R^d)$, we have
\begin{equation}\label{eq:horizontal}
\lim_{\e\rightarrow 0} \cF_\H(g\otimes \chi(\e\cdot)) = 2\pi (\cG_\H g)  \mu_{\wh\H_0^d}
\end{equation}
in the sense of measures on~$\wh\H^d,$  where~$\mu_{\wh\H_0^d}$ is the measure (supported in $\wh \H_0^d$) defined for all    continuous compactly supported functions~$\theta$ 
on~$\wh\H^d$ by
\beq
\label {limitmeasureeq1}
\langle \mu_{\wh \H^d_0} ,\theta \rangle = \int_{\wh \H^d_0}  \theta (\dot x,k) \,d\mu_{\wh \H^d_0}(\dot x,k)
\eqdefa2^{-d}\sum_{k\in \Z^d}  \biggl( \int_{(\R_{-})^d} \theta(\dot x,k)\,d\dot x +  \int_{(\R_{+})^d} \theta(\dot x,k)\,d\dot x\biggr)\cdotp
\eeq
}
\end{theorem}
The above  theorem allows to give  a meaning of 
 the Fourier transform of a smooth function that does not depend on the vertical variable.
 The next step would be to study whether our approach allows, as in the Euclidean case,
 to extend the definition of the Fourier transform to a much larger 
set of functions, or even to tempered distributions.   This requires  a fine characterization of~$\cF_\H(\cS(\H^d))$ the range  of~$\cS(\H^d)$ by~$\cF_\H,$ which  will be the purpose  of a forthcoming paper
\cite{bcdh}.
 
 \medbreak
We end this section with a short description of  the structure of the rest of the paper, and  of 
the main ideas of the proofs.

Section\refer {ProofCFL1}
 is devoted to the proof of the first part of Theorem\refer {FourierL1basic}.  It relies on the fact that  the function~$\cW(\cdot,Y)$  is uniformly continuous (for distance $\wh d$) on bounded sets of~$\wt \H^d,$  
and can thus be extended  to the closure~$\wh \H^d$  of~$\wt \H^d$.  Establishing 
that property requires our using an explicit  asymptotic expansion of~$\cW.$ 

Proving Theorem\refer {Fourierhorizontal}  is the purpose of Section\refer {FourierHorizontal}. The main two ingredients   are the following ones. First, we show that   if~$\psi$ is an integrable function on~$\R$ with integral~$1,$ then we have
$$
\lim_{\e\rightarrow 0}  \frac 1 \e  \psi \Big(\frac \lam \e\Bigr) d\wh w= \mu_{\wh \H^d_0},
$$
in the sense of measures on $\wh\H^d.$ That is to say,  for any continuous compactly supported function $\theta$ on $\wh\H^d,$ we have
\beq
\label {limsimplecoucheH_0}
\lim_{\e\rightarrow 0} \int_{\wh \H^d} \frac 1 \e \psi \Big(\frac \lam \e\Bigr) \theta (\wh w)\, d\wh w
= \langle \mu_{\wh \H^d_0} ,\theta_{|\wh \H^d_0} \rangle.
\eeq
Then  by a density argument,  the proof of Theorem\refer{Fourierhorizontal} 
reduces to the case when~$g$ is in $\cS(T^\star\R^d)$.
\smallbreak
Section\refer {proofFormulacK} is devoted to computing~$\cK$. This will be based
on the following properties (that we shall first establish):
\begin{itemize}
\item $\cK(0,k,Y)=\delta_{0,k}$ for all $Y$ in $T^\star\R;$\smallbreak
\item The symmetry identities:
\begin{equation}\label{eq:Ksym}
\begin{aligned}
&\cK(\dot x,-k,-Y)=\overline{\cK(\dot x,k,Y)},\qquad \cK(-\dot x,-k,Y)=(-1)^k\cK(\dot x,k,Y)\\
&\andf \cK(-\dot x,k,Y)=\overline{\cK(\dot x,k,Y)};\end{aligned}
\end{equation}
\item The identity
\begin{equation}
\label {eq:KLap}
\D_{Y} \cK(\dot x,k,Y) = -4|\dot x| \cK(\dot x,k,Y);
\end{equation}
\item The relation 
\begin{equation}
\label {eq:Kk}
ik \cK (\dot x,k,Y) = \bigl(\eta\partial_y\cK(\dot x,k,Y) -y \partial_\eta \cK(\dot x,k,Y)\bigr)\,\sgn(\dot x);
\end{equation}
\item
The convolution property
\beq
\label {Convollam=0}
\cK(\dot x,k,Y_1+Y_2) = \sum_{k'\in \ZZ} \cK(\dot x,k-k',Y_1)\cK(\dot x,k',Y_2);
\eeq
\item
And finally, the following relation for $\dot x > 0$ given by the study of~$\cF_\H (|Y|^2f)$:
\beq
\label  {Y2FouriercK}
|Y|^2\cK + \dot x\partial_{\dot x} ^2 \cK +\partial_{\dot x} \cK -\frac {k^2} {4\dot x} \cK=0.
\eeq
\end{itemize}
Let us emphasize that  proving first\refeq  {limsimplecoucheH_0} 
is essential to justify rigorously\refeq {FourierL1basiceq2b}. 
\medbreak
Finally, Section\refer  {StutyofcGH} is devoted to the proof of 
an inversion formula involving Operator~$\cG_\H.$ 
 Some basic properties of Hermite functions and  of Wigner transform of Hermite functions
 are recalled in Appendix. There, we also prove the decay result
 stated in Lemma\refer{decaylambdan}. 


\section {The uniform continuity of the Fourier transform of an~$L^1$ function }
\label {ProofCFL1}

The key  to the proof of Theorem \ref{FourierL1basic}   is a refined study of the behavior of 
functions~$\cW(\cdot,Y)$ defined by\refeq{definWigner} on the set~$\wt \H^d.$ Of course, a special attention will be given
 to the neighborhood of $\wh\H^d_0.$
This is the aim of the following proposition.
\begin{proposition}
\label {ProofCFL1_Prop1}
{\sl Let~$R_0$ be a positive real number, and let
$$
 \cB(R_0)\eqdefa\Bigl\{(n,m,\lambda) \in\wt \H^d,\ 
 |\lambda|(|n+m|+d) +|n-m|  \leq R_0\Bigr\} 
 \times\Bigl\{ Y\in T^\star\R^d,\  |Y|\leq R_0\Bigr\}\cdotp
 $$
 The function~$\cW(\cdot,Y)$ restricted to $\cB(R_0)$  is uniformly continuous   with respect to~$\wh w,$ that is 
 $$
 \forall \e>0\,,\ \exists \al_\e>0\,,\ \forall (\wh w_j,Y)\in \cB(R_0)\,, \ 
 \wh d(\wh w_1,\wh w_2)<\al_\e\Longrightarrow \ \bigl | \cW(\wh w_1,Y) -\cW(\wh w_2,Y)\bigr| <\e.
 $$
 Furthermore, for any $(\dot x,k)$ in~$\wh\H^d_0,$ we have 
 $$
 \lim_{\wh w\to(\dot x,k)} \cW(\wh w,Y)=\cK_d(\dot x,k,Y)
 $$
 where the function $\cK_d$ is defined on $\wh\H^d_0\times T^\star\R^d$ by
 \beq
 \label {definPhaseFlambda=0}
 \begin{aligned}
 & \cK_d(\dot x, k,Y) \eqdefa  \sum_{(\ell_1,\ell_2)\in\N^d\!\times\!\N^d}\frac {(i\eta)^{\ell_1}} {\ell_1! }\, \frac {y^{\ell_2}} {\ell_2!}  \, F_{\ell_1,\ell_2} (k) \, (\sgn\dot x)^{\ell_1} |\dot x|^{\frac{\ell_1+\ell_2}2}
  \with \\
  &F_{\ell_1,\ell_2} (k) \eqdefa  \sumetage{\ell_1'\leq \ell_1, \ell'_2\leq \ell_2}  {k+\ell_1-2\ell'_1= \ell_2-2\ell'_2} (-1)^{\ell_2-\ell'_2} \begin{pmatrix} \ell_1 \\ \ell' _1\end{pmatrix}  \begin{pmatrix} \ell_2 \\ \ell' _2\end{pmatrix}.
  \end{aligned}
 \eeq 
 Above, $\sgn \dot x$ designates the (common) sign of all components of $\dot x,$
 and $|\dot x|\eqdefa(|\dot x_1|,\cdots,|\dot x_d|).$}
\end{proposition}
\begin{proof}
Let us first perform  the change of variable~$z'= -y+z$ in\refeq{definWigner} so as to get
\beq
\label {ProofCFL1_Prop1demoeq1}
\begin{aligned}
\cW(\wh w,Y) & = e^{2i\lam\langle \eta,y\rangle} \wt \cW(\wh w,Y) \with \\
\wt \cW(\wh w,Y) & \eqdefa \int_{\R^d} e^{2i\lam\langle \eta,z'\rangle} H_{n,\lam} (2y+z') H_{m,\lam} (z') \,dz'.
\end{aligned}
\eeq

Obviously,  the uniform continuity of~$\cW$ reduces to that  of~$\wt\cW$. Moreover, as the integral defining~$\wt \cW$ is a product of~$d$ integrals on~$\R$ (of modulus
bounded by $1$),  it is enough to study the one dimensional case.   

Let us start with   the case where both $\wh w_1=(n_1,m_1,\lam)$ 
and $\wh w_2=(n_2,m_2,\lam)$ are relatively far away from~$\wh\H_0^1$.  
As we need only to consider the situation where $\wh w_1$ and $\wh w_2$ are close 
to one another,  one may assume that $(n_1,m_1)=(n_2,m_2)=(n,m).$ 
Then we can write that
  \beq
\label {ProofCFL1_Prop1demoeq2}
\begin{aligned}
\wt \cW(\wh w_1,Y)- \wt \cW(\wh w_2,Y)& = \int_\R \bigl( e^{2i\lam_1 \eta z} - e^{2i\lam_2 \eta z}\bigr)  H_{n,\lam_1} (2y+z) H_{m,\lam_1} (z) \,dz\\
& {} +  \int_\R  e^{2i\lam_2 \eta z}  \bigl( H_{n,\lam_1} (2y+z)- H_{n,\lam_2} (2y+z)\bigr)  
H_{m,\lam_1} (z) \,dz\\
& {} +\int_\R  e^{2i\lam_2 \eta z} H_{n,\lam_2} (2y+z) \bigl( H_{m,\lam_1} (z)-H_{m,\lam_2} (z)\bigr) \,dz.
\end{aligned}
\eeq
Clearly, we have
\begin{equation}\label{eq:Wligne1}
\biggl|\int_\R \bigl( e^{2i\lam_1 \eta z} - e^{2i\lam_2 \eta z}\bigr)  H_{n,\lam_1} (2y+z) H_{m,\lam_1} (z) \,dz\biggr|
\leq2|\lam_1|^{-\frac12} R_0|\lam_2-\lam_1|\,\|MH_m\|_{L^2}.
\end{equation}

Next, let us study the continuity of  the map
$\lam \longmapsto  H_{n, \lam}$ in the case $d=1.$ One may  write
\beno
H_{n,\lam_1} (x) -H_{n,\lam_2} (x) & =  & \bigl(|\lam_1|^{\frac 14} -|\lam_2|^{\frac 1 4}\bigr) H_n (|\lam_1|^{\frac 12} x) 
+|\lam_2|^{\frac 1 4}\bigl(H_n (|\lam_1|^{\frac 12}x) -H_n(|\lam_2|^{\frac 12}x) \bigr)\\
& = & \bigl(|\lam_1|^{\frac 14} -|\lam_2|^{\frac 1 4}\bigr) H_n (|\lam_1|^{\frac 12} x) \\
&& \qquad{}+ |\lam_2|^{\frac 1 4}(|\lam_1|^{\frac 12} -|\lam_2|^{\frac 12}) 
\, x\!\int_0^1 H'_n\bigl(( |\lam_2|^{\frac 12} +t(|\lam_1|^{\frac 12} -|\lam_2|^{\frac 12} ))x\bigr)dt.
\eeno
 If~$\bigl|| \lam_1|^{\frac 12} -|\lam_2|^{\frac 12}\bigr|\leq \ds \frac 1 2 |\lam_2|^{\frac12},$
 then the changes of variable
 $ x'=|\lam_1|^{\frac 12} x$ and $ x'= ( |\lam_2|^{\frac 12} +t(|\lam_1|^{\frac 12} -|\lam_2|^{\frac 12} ))x,$
 respectively,   together with the fact that the Hermite functions have~$L^2$ norms equal to~$1$ ensure that 
  \beq
\label {ProofCFL1_Prop1demoeq3}
\|H_{n,\lam_1} -H_{n,\lam_2} \|_{L^2} \leq \frac {\bigl| |\lam_1|^{\frac 14} -|\lam_2|^{\frac 1 4}\bigr|} {|\lam_1|^{\frac 14}} + 4 \frac {\bigl|| \lam_1|^{\frac 12} -|\lam_2|^{\frac 12} \bigr|}{|\lam_2|^{\frac 12}} \,\|MH_n'\|_{L^2}.
\eeq
Using\refeq  {relationsHHermiteCAb}, we get that
$$
MH'_n = \frac 12 \bigl( \sqrt {n(n-1)}H_{n-2} -H_n-\sqrt { (n+1)(n+2)} H_{n+2}\bigr).
$$
As the family of Hermite functions is  an orthonormal basis of~$L^2$,  one can write that
$$
4 \|MH_n'\|_{L^2}^2 = n(n-1) +1+(n+1)(n+2) = 2n^2+2n+3.
$$
Combining with\refeq  {ProofCFL1_Prop1demoeq2} and\refeq  {ProofCFL1_Prop1demoeq3},  we conclude that
if $\bigl| | \lam_1| -|\lam_2||\leq \frac 1 2  \,|\lam_2|^{\frac12} \bigl( |\lam_1|^{\frac 12} +|\lam_2|^{\frac 12}\bigr)$
and $(|\lam_1|+|\lam_2|) |n+m|+|\lam_1|+|\lam_2|+|Y|\leq R_0$ then
  \begin{equation}\label {ProofCFL1_Prop1demoeq4}
\bigl |\wt  \cW(\wh w_1,Y)- \wt \cW(\wh w_2,Y)\bigr| \leq C(R_0)|\lam_1-\lam_2|\biggl(\frac1{|\lam_1|} +\frac 1{\lam_1^2}\biggr)\cdotp 
\end{equation}
That estimate fails  if  the above condition on~$\bigl| |\lam_1| -|\lam_2|\bigr|$  is not satisfied.  To overcome that difficulty,  we need  the following lemma.
\begin{lemma}
\label {Wignerconvnormally}
{\sl
 The series 
$$
\sum_{\ell_1,\ell_2}(\sgn \lam)^{\ell_1}   |\lam|^{\frac {\ell_1+\ell_2} 2} \,
\frac{(2i\eta)^{\ell_1}(2y)^{\ell_2}} {\ell_1! \ell_2!}  
\bigl(M^{\ell_1}H_{m} | \partial^{\ell_2}H_{n}\bigr)_{L^2(\R^d)}
 $$
converges normally towards $ \wt \cW$ on $\cB(R_{0})$.
}
\end{lemma}
\begin{proof}
Again, as Hermite functions in dimension~$d$  are tensor products of one-dimensional Hermite functions, it is enough to prove the above lemma in dimension~$1$.
Now, using  the expansion of the exponential function and  Lebesgue theorem, we get that for any fixed~$(\wh w,Y)$ in~$\wt \H^1\times T^\star \R,$
\ben
\nonumber
\wt \cW(\wh w,Y) & =  & \sum_{\ell_1=0}^\infty \frac 1 {\ell_1!} (2i\lam\eta)^{\ell_1} \int_{\R} 
 H_{n,\lam} (2y+z)  z^{\ell_1}  H_{m,\lam} (z) \,dz\\
\label {Wignerconvnormallydemoeq2} 
& = & \sum_{\ell_1=0}^\infty (\sgn\lambda)^{\ell_1} \frac{(2i\eta)^{\ell_1}}{\ell_1!} \, |\lam|^{\frac {\ell_1}2} \int_{\R} 
 H_{n,\lam} (2y+z) (M^{\ell_1} H_m)_\lam(z) \,dz.
 \een
Let  us prove that the series converges for  the supremum norm on~$\cB(R_0)$.  Clearly,\refeq {relationsHHermiteCAb}
implies that for all integers $\ell\geq1$ and $m$ in~$\N,$
$$
\|(\sqrt2M)^\ell H_m\|_{L^2(\R)}\leq\sqrt m\,\|(\sqrt2M)^{\ell-1} H_{m-1}\|_{L^2(\R)}
+\sqrt{m+1}\,\|(\sqrt2M)^{\ell-1} H_{m+1}\|_{L^2(\R)},
$$
which, by an obvious induction yields for all $(\ell_1,m)$ in~$\N^2,$
\beq
\label {Wignerconvnormallydemoeq3}
\|M^{\ell_1}H_{m}\|_{L^2(\R)}   \leq   (2m+2\ell_1)^{\frac{\ell_1}2}. 
\eeq
Hence  the generic term of the series of\refeq {Wignerconvnormallydemoeq2}  can be bounded by:
$$
W_{\ell_1} (\wh w) \eqdefa \frac { (2\sqrt 2R_0)^{\ell_1}  } {\ell_1!} |\lam|^{\frac {\ell_1} 2}
(m+\ell_1)^{\frac{\ell_1}2}.
$$
Let us observe that, because~$|\lam| m$ and~$|\lam|$ are less than~$R_0$, we have
\beno
\frac { W_{\ell_1+1} (\wh w) } {W_{\ell_1} (\wh w) } &=  &  \frac{2\sqrt2 R_0}{\ell_1+1}
\sqrt{|\lam|(m+\ell_1+1)}\biggl(1+\frac{1}{m+\ell_1}\biggr)^{\frac{\ell_1}2}  \\
& \leq &  \frac {2\sqrt{2e}\,R_0 } {\ell_1+1}  \sqrt {R_0} \bigl( 1 +\sqrt {\ell_1+1}\bigr).
\eeno
This implies that the series converges with respect to the supremum norm on~$\cB(R_0)$.
\smallbreak
Next, for fixed~$\ell_1,$  we want to expand 
$$
 |\lam|^{\frac {\ell_1}2}\int_{\R}  H_{n,\lam} (2y+z) (M^{\ell_1} H_m)_\lam(z) \,dz
$$
as a series with respect to the variable $y.$ To this end, we just have to expand the real analytic Hermite functions
as follows:  
$$
H_{n,\lam} (z+2y) = \sum_{\ell_2=0} ^\infty \frac { (2y)^{\ell_2} } {\ell_2! }  |\lam|^{\frac {\ell_2} 2} ( H_n^{(\ell_2)})_\lam(z).
$$
Then we have to study (for fixed~$\ell_1$) the convergence of the series with general term, 
$$
W_{\ell_1,\ell_2}(\wh w,Y) \eqdefa  \frac { (2y)^{\ell_2} } {\ell_2! }  |\lam|^{\frac {\ell_2} 2} \bigl(H_n^{(\ell_2)} | M^{\ell_1} H_m\bigr)_{L^2}.
$$
Using again\refeq {relationsHHermiteCAb}, we see that 
$$
\|H_{n}^{(\ell_2)}\|_{L^2(\R)} \leq   (2n+2\ell_2)^{\frac{\ell_2}2}. 
$$
Hence, arguing as above, we get  for any~$(\wh w,Y)$ in~$\cB(R_0)$, 
$$
W_{\ell_1,\ell_2}(\wh w,Y)  \leq  2^{\frac {\ell_1} 2}  (m\!+\!\ell_1)^{\frac {\ell_1} 2}  \wt W_{\ell_2}(\wh w,Y) \with 
  \wt W_{\ell_2}(\wh w,Y)  \eqdefa 
\frac { (2\sqrt 2R_0)^{\ell_2} } {\ell_2! }  |\lam|^{\frac {\ell_2} 2}  (n\!+\!\ell_2)^{\frac{\ell_2}2},
$$
and it is now easy to complete  the proof of the lemma.
\end{proof}

\medbreak
Reverting to the proof of the continuity of $\wt\cW$ in the neighborhood of $\wh\H^d_0,$ 
 the problem now consists in investigating the behavior  of the function
$$
\cH_{\ell_1,\ell_2}:\quad \left\{
\begin{array} {rcl}
\wt \H^1 & \longrightarrow &\R\\
\wh w=(n,m,\lam) & \longmapsto  &   |\lam|^{\frac{\ell_1+\ell_2}2 } \,
\bigl(M^{\ell_1}H_{m} |H_{n}^{(\ell_2)}\bigr)_{L^2(\R)}
\end{array}
\right.
$$
when~$\lam$ tends to~$0$ and $\lam(n+m)\to\dot x$ for fixed $k\eqdefa m-n.$ 
\medbreak
{}From  Relations \eqref{Mjdj}, we infer that  
$$
\cH_{\ell_1,\ell_2}(\wh w)  =   2^{-(\ell_1+\ell_2)}\,  |\lam|^{\frac{\ell_1+\ell_2}2 }  
 \bigl((A+C)^{\ell_1}H_{m} | (A-C)^{\ell_2}H_{n} \bigr)_{L^2(\R)}.
$$
The explicit computation of~$(A\pm C)^{\ell}$ is doable but tedious and fortunately turns out to be useless when~$\lam$ tends to~$0$.  Indeed, we have the following lemma: 
\begin{lemma}
\label {ProofCFL1lemma2}
{\sl
A constant~$C_\ell(R_0)$ (depending only on~$R_0$ and~$\ell$) exists  such that, for any~$(n,\lam)$ 
with $\lam>0$ and~$\lam n \leq R_0$,  we have (with the convention that $H_p=0$ if $p<0$):
$$
\Bigl \| \lam^{\frac  \ell 2 } \Bigl(\frac{A \pm C}2\Bigr)^{\ell } H_n - \Bigl(\frac{\lam n}2\Bigr)^{\frac  \ell 2} \sum_{\ell'=0} ^\ell  (\pm1)^{\ell-\ell'} \begin{pmatrix} \ell \\ \ell' \end{pmatrix}  H_{n+\ell-2\ell'} \Bigr\|_{L^2(\R)} \leq C_\ell(R_0)\lam^{\frac 12}.
$$
}
\end{lemma}
\begin{proof}
Let~$\cV_{n,\ell}$ be the vector space generated by~$(H_{n+\ell'})_{-\ell\leq \ell'\leq \ell},$
 equipped with the~$L^2(\R)$-norm. 
 Let 
 $$
  R_{n,\ell}  \eqdefa  \lam^{\frac  \ell 2 } (A \pm C)^{\ell } H_n - (2\lam n)^{\frac  \ell 2} \sum_{\ell'=0} ^\ell  (\pm 1)^{\ell-\ell'} \begin{pmatrix} \ell \\ \ell' \end{pmatrix}  H_{n+\ell-2\ell'}.
  $$
  Formulae \eqref{relationsHHermiteCA} guarantee that  $R_{n,\ell}$ is in~$\cV_{n,\ell}.$  
 Let us now prove  by induction on $\ell$ that
\begin{equation}
\label {ProofCFL1lemma2demoeq1}
\|R_{n,\ell}\|_{\cV_{n,\ell}} \leq  C_\ell(R_0)\lam^{\frac 12}.
 \end{equation}
In the case when~$\ell$ equals~$1$, by definition of~$A$ and~$C$, we have
\beno
\lam^{\frac 12} (A\pm C)  H_n &= & \lam^{\frac 12} \bigl( \sqrt {2n} H_{n-1} \pm \sqrt {2n+2} H_{n+1} \bigr)\\
& =& \sqrt {2\lam n} (H_{n-1} \pm H_{n+1}) \pm \frac {2\sqrt \lam} {\sqrt {2n+2} +\sqrt {2n}} H_{n+1}
\eeno
and\refeq {ProofCFL1lemma2demoeq1} is thus obvious.
\smallbreak
Let us now observe that, for any~$\ell'$ in~$\{-\ell,\cdots ,\ell\}$, we have
$$\begin{aligned}
\lam^{\frac 12} \|A H_{n+\ell'} \|_{L^2(\R)}  &=  \sqrt {2\lam(n+\ell')}\,  \|H_{n+\ell'-1}\|_{L^2(\R)} \andf\\
\lam^{\frac 12} \|C H_{n+\ell'} \|_{L^2(\R)}  & = \sqrt {2\lam(n+\ell'+1)} \, \|H_{n+\ell'+1}\|_{L^2(\R)}.
\end{aligned}
$$
This gives that for all $\lam(n+1)\leq R_0,$
\beq
\label {ProofCFL1lemma2demoeq2}
\bigl \| \lam^{\frac 12} (A\pm C)\bigr\|_{\cL(\cV_{n,\ell} ; \cV_{n,\ell+1})} \leq C_\ell(R_0).
\eeq
Let us assume that\refeq  {ProofCFL1lemma2demoeq1} holds for some~$\ell$.  Inequality\refeq {ProofCFL1lemma2demoeq2} implies that 
\beq
\label {ProofCFL1lemma2demoeq3}
\bigl \| \lam^{\frac 12} (A\pm C)R_{n,\ell} \bigr\|_{\cV_{n,\ell+1}} \leq \lam^{\frac 12}C_\ell(R_0).
\eeq
Then, for any~$\ell'$ in~$\{0,\cdots,\ell\}$, we have
$$
\begin{aligned}
\lam^{\frac 12} (A\pm C)  H_{n+\ell-2\ell'}  &=  \lam^{\frac 12} \bigl( \sqrt {2n+2\ell-4\ell'}\, H_{n+\ell-2\ell'-1} \pm \sqrt {2n+2\ell-4\ell'+2}\,  H_{n+\ell-2\ell'+1} \bigr)\\
& = \sqrt {2\lam n} \bigl(H_{n+\ell+1-2(\ell'+1)} \pm H_{n+\ell+1-2\ell'} \bigr) \\
&\hspace{-2cm}+\frac {2\lam^{\frac 12}(\ell-2\ell')} {\sqrt {2n+2\ell-4\ell'}+ \sqrt {2n}}H_{n+\ell-2\ell'-1} \pm \frac {2\lam^{\frac 12}(\ell-2\ell'+1)} {\sqrt {2n+2\ell-4\ell'+2}+ \sqrt {2n}}H_{n+\ell-2\ell'+1}\,\cdotp
\end{aligned}
$$
We deduce that for any~$\ell'$ in~$\{0, \cdots, \ell\}$, 
$$
\bigl\| \lam^{\frac 12} (A\pm C)  H_{n+\ell-2\ell'} 
-  \sqrt {2\lam n} \bigl(H_{n+\ell+1-2(\ell'+1)} \pm H_{n+\ell+1-2\ell'} \bigr)\bigr\|_{\cV_{n,\ell+1} }
\leq C_{\ell+1} (R_0)\lam^{\frac 12} .
$$
Using\refeq{ProofCFL1lemma2demoeq3} gives
$$\displaylines{\quad
\bigl \| \lam^{\frac  {\ell+1}  2 } (A \pm C)^{\ell +1 } H_n - (2\lam n)^{\frac { \ell+1} 2} \Sigma_{n,\ell}
 \bigr\|_{L^2(\R)} 
\leq C_{\ell+1} (R_0)\lam^{\frac 12}\hfill\cr\hfill
\with \Sigma_{n,\ell}  \eqdefa  \sum_{\ell'=0} ^\ell  (\pm1)^{\ell-\ell'} \begin{pmatrix} \ell \\ \ell' \end{pmatrix}   \bigl(H_{n+\ell+1-2(\ell'+1)} \pm H_{n+\ell+1-2\ell'} \bigr).\quad}
$$
Now, Pascal's rule ensures that
$$\begin{aligned}
\Sigma_{n,\ell}
 = & \sum_{\ell'=1} ^{\ell+1}  (\pm1)^{\ell+1-\ell'} \begin{pmatrix} \ell \\ \ell'-1 \end{pmatrix}  H_{n+\ell+1-2\ell'} 
+\sum_{\ell'=0} ^\ell  (\pm1)^{\ell+1-\ell'} \begin{pmatrix} \ell \\ \ell' \end{pmatrix}  H_{n+\ell+1-2\ell'} \\
 = & \sum_{\ell'=0} ^{\ell+1}  (\pm1)^{\ell+1-\ell'} \begin{pmatrix} \ell+1 \\ \ell' \end{pmatrix}  H_{n+\ell+1-2\ell'}. 
\end{aligned}$$
The lemma is proved.
\end{proof}

\medbreak
  
From this lemma, we can deduce the following corollary.
\begin{cor}
\label {ProofCFL1_Coroll1}
{\sl
For any $(\ell_1,\ell_2)$ in $\N^2$  and $R_0>0,$ there exists a constant~�$C_{\ell_1,\ell_2} (R_0)$  such that for all~$(n,n+k,\lam)$ in $\wt \H^1$ with~$|\lam n| +|k|+ |\lam|\leq R_0$, we have
$$\displaylines{\qquad
\Bigl|\cH_{\ell_1,\ell_2}(\wh w) -  F_{\ell_1,\ell_2} (k)  \Bigl(\frac{|\lam| n}2\Bigr)^{\frac {\ell_1+\ell_2}2}\Bigr| 
 \leq    C_{\ell_1,\ell_2}(R_0) |\lam| ^{\frac 12}\hfill\cr\hfill
\with  F_{\ell_1,\ell_2} (k) \eqdefa  \sumetage{\ell_1'\leq \ell_1, \ell'_2\leq \ell_2}  {k+\ell_1-2\ell'_1= \ell_2-2\ell'_2} (-1)^{\ell_2-\ell'_2} \begin{pmatrix} \ell_1 \\ \ell' _1\end{pmatrix}  \begin{pmatrix} \ell_2 \\ \ell' _2\end{pmatrix}\cdotp\qquad}$$
}
\end{cor}
\begin{proof}
Lemma\refer {ProofCFL1lemma2} implies that 
$$
\displaylines{
\Bigl|\cH_{\ell_1,\ell_2}(\wh w) -  \Bigl(\frac{|\lam| n}2\Bigr)^{\frac {\ell_2}2} \Bigl(\frac{|\lam| (n+k)}2\Bigr)^{\frac { \ell_1} 2} \sumetage{\ell_1'\leq \ell_1} 
{\ell'_2\leq \ell_2} (-1)^{\ell_2-\ell'_2} \begin{pmatrix} \ell_1 \\ \ell' _1\end{pmatrix}  \begin{pmatrix} \ell_2 \\ \ell' _2\end{pmatrix} \bigl(H_{n+k+\ell_1-2\ell'_1} |H_{n+\ell_2-2\ell'_2}\bigr )_{L^2}\Bigr|\cr
{}
 \leq  C_{\ell_1,\ell_2}(R_0) |\lam| ^{\frac 12}.
}
$$
Now, let us notice that 
$$
(|\lam|(n+k))^{\frac{\ell_1}2}-(|\lam|n)^{\frac{\ell_1}2}=\frac{|\lam| k}{\sqrt{|\lam| n}+\sqrt{|\lam|(n+k)}}
\sum_{\ell_1'=0}^{\ell_1-1}\sqrt{|\lam| n}^{\ell_1'}\sqrt{|\lam|(n+k)}^{\ell_1-1-\ell_1'}.
$$
Hence it is clear that for fixed~$k$ in~$\Z$ such that $|k|\leq R_0$, we have, for~$|\lam|\leq R_0$ and~$|n\lam|\leq R_0$, 
$$
\bigl |  (|\lam| n)^{\frac {\ell_2}2} (|\lam| (n+k))^{\frac { \ell_1} 2} - |\lam n|^{\frac {\ell_1+\ell_2}2} \bigr| \leq C_{\ell_1,\ell_2} (R_0) |\lam|^{\frac 12}.
$$
Thanks to \eqref{def:kro}, we conclude the proof.
\end{proof}

\medbreak
\noindent {\it Conclusion of the proof of Proposition\refer {ProofCFL1_Prop1} } Consider a positive real number~$\e$. 
Recall that
$$
\wt\cW(\wh w,Y)=\sum_{\ell_1,\ell_2}(\sgn\lambda)^{\ell_1} \frac{(2i\eta)^{\ell_1}(2y)^{\ell_2}}{\ell_1!\ell_2!}
\cH_{\ell_1,\ell_2}(\wh w).
$$
Clearly, it suffices to prove the uniform continuity of $\wt\cW$ for all subset of $\wh\H^d$ corresponding
to some  \emph{fixed} value $k$ of $m-n.$ Now, considering $\wh w_1=(n_1,n_1+k,\lam_1)$ and $\wh w_2=(n_2,n_2+k,\lam_2),$
Lemma\refer  {Wignerconvnormally} implies that for all $\ep>0,$ there exist two integers~$L_ {1,\e}$ and $L_{2,\e}$ such that 
\beq
\label {ProofCFL1_Prop1demoeq6} 
\begin{aligned}
&\bigl |\wt  \cW(\wh w_1,Y)- \wt \cW(\wh w_2,Y)\bigr|  \leq  \frac \e 4
+\sumetage {\ell_1\leq L_{1,\e}} {\ell_2\leq L_{2,\e}}   \frac {(2|\eta|)^{\ell_1}(2|y|)^{\ell_2}} {\ell_1!\ell_2!}   \\
&\qquad{}
\times 
\bigl | (\sgn \lam_1)^{\ell_1} \cH_{\ell_1,\ell_2} (n_1,n_1+k,\lam_1) - (\sgn \lam_2)^{\ell_1} \cH_{\ell_1,\ell_2} (n_2,n_2+k,\lam_2) \bigr| \,.
\end{aligned}
\eeq
Let~$C_\e(R_0)$  be the supremum for~$\ell_1\leq L_{1,\e}$ and   $\ell_2\leq L_{2,\e}$  of all constants~$C_{\ell_1,\ell_2} (R_0)$ which appear in Corollary\refer {ProofCFL1_Coroll1}.
Then we have 
\begin{equation}
\label {ProofCFL1_Prop1demoeq6b}
\begin{aligned}
&|\lam_1|\!+\!|\lam_2|\leq A(\e,R_0) \Longrightarrow
\bigl |\wt  \cW(\wh w_1,Y)- \wt \cW(\wh w_2,Y)\bigr|  \leq  \frac \e 2
+\!\!\!\sumetage {\ell_1\leq L_{1,\e}} {\ell_2\leq L_{2,\e}} \!\!\!\!\frac {(2\,R_0)^{\ell_1+\ell_2}} {\ell_1!\ell_2!}|F_{\ell_1,\ell_2} (k) | \\
&\qquad\qquad\qquad\qquad\qquad\qquad{} \times 
\Bigl |(\sgn \lam_1)^{\ell_1}  \Big|\frac{\lam_1 n_1}2\Big|^{\frac {\ell_1+\ell_2}2} -  (\sgn \lam_2)^{\ell_1}\Big|\frac{\lam_2 n_2}2\Big|^{\frac {\ell_1+\ell_2}2} \Bigl |\\
&\qquad\qquad\qquad\qquad\qquad\qquad\qquad\qquad\qquad\qquad\qquad \with  A(\e,R_0)\eqdefa \ds \frac {e^{-8R_0} \e^2} {32 C_\e^2(R_0)}\,\cdotp
\end{aligned}
\end{equation}
If~$\ell_1+\ell_2=0$ then the last term of the above inequality is~$0$. If~$\ell_1+\ell_2$ is positive,  as~$|F_{\ell_1,\ell_2} (k)|$ is less than~$2^{\ell_1+\ell_2}$,  we have, using\refeq  {ProofCFL1_Prop1demoeq6b},
\beq
\label {ProofCFL1_Prop1demoeq7} 
\begin{aligned}
& |\lam_1|+|\lam_2|\leq A(\e,R_0) \andf |\lam_1 n_1|+|\lam_2 n_2| \leq \frac 1{16} \e^2 e^{-8R_0}\\
& \qquad\qquad\qquad\qquad\qquad\qquad\qquad\qquad\qquad \Longrightarrow 
\bigl |\wt  \cW(\wh w_1,Y)- \wt \cW(\wh w_2,Y)\bigr|  \leq  \e.
\end{aligned}
\eeq
In the case when~$ |\lam_1 n_1|+|\lam_2 n_2| $ is greater than~$\ds \frac 1{16}\e^2 e^{-8R_0}$ then if
$$
\bigl| \lam_1n_1-\lam_2 n_2\bigr| \leq \frac 1{32}\e^2 e^{-8R_0}
$$
then~$\lam_1$ and~$\lam_2$  have the same sign.
The sum in the right-hand side term is  finite, 
and it is clear that each term converges uniformly to $0$  if $\lambda_2n_2$ tends to $\lambda_1 n_1.$
 Thus a positive real number~$\eta_\e$ exists such that 
\beq
\label {ProofCFL1_Prop1demoeq8}
|\lam_1|+|\lam_2|\leq A(\e,R_0) \andf |\lam_1n_1-\lam_2n_2| \leq \eta_\e \Longrightarrow
\bigl |\wt  \cW(\wh w_1,Y)- \wt \cW(\wh w_2,Y)\bigr|  \leq  \e.\\
\eeq
Finally, we have to consider the case where~$|\lam_1|+|\lam_2|\geq  A(\e,R_0)$. With no loss of generality, one can assume
that~$\lam_2 \geq \ds \frac 12 A(\e,R_0)$. Thus, if~$|\lam_1-\lam_2|$ is less than~$\ds \frac 14 A(\e,R_0)$ we have~$\lam_1 \geq \ds \frac 14 A(\e,R_0)$ and we can apply Inequality\refeq {ProofCFL1_Prop1demoeq4} which gives 
(supposing that $A(\ep,R_0)\leq1$):
$$
\bigl |\wt  \cW(\wh w_1,Y)- \wt \cW(\wh w_2,Y)\bigr|  \leq 2C(R_0) \Bigl( \frac 1 4 A(\e,R_0)\Bigr)^{-2} |\lam_1-\lam_2|. 
$$
Together with\refeq {ProofCFL1_Prop1demoeq7}, this gives, if~$(n_j,m_j,\lam_j)$ are in~$\cB(R_0)$,
$$
|\lam_1-\lam_2| \leq \frac {\e A^2(\e,R_0)} {32C(R_0)} \andf |\lam_1n_1-\lam_2 n_2| <\eta_\e
\Longrightarrow \bigl |\wt  \cW(\wh w_1,Y)- \wt \cW(\wh w_2,Y)\bigr| <\e.
$$
The proposition is proved.
\end{proof}

\medbreak
\begin{proof}  [End of the proof of the first part of Theorem\refer {FourierL1basic}]
Because of the integrability of~$f$, Proposition\refer {ProofCFL1_Prop1} 
implies that~$\wh f_\H$ is uniformly continuous on~$\wt \H^d,$ and  can thus be extended to a uniformly continuous function on the complete metric space~$\wh\H^d.$

Let us finally  establish that 
$\wh f_\H(\wh w) \to0$ when $\wh w$ goes to infinity. In the case where $f$ is in $\cS(\H^d),$
this in an obvious consequence of Lemma\refer {decaylambdan}.
The general case of an integrable function on $\H^d$ follows by density as, obviously, 
Formula \eqref{definFourierWigner}
implies that the map  $f\to \wh f_\H$ 
is continuous from $L^1(\H^d)$ to $L^\infty(\wh \H^d).$ \end{proof}
 \medbreak 
 We are now ready to  establish Formula\refeq {FourierL1basiceq2} for any integrable function $f$ on $\H^d.$ 
 So let us fix some~$(\dot x,k)$ in~$\wh\H^d_0$, and consider 
 a sequence~$\suite {\wh w } p \N  = (n_p,n_p+k,\lam_p)_{p\in \N}$ such that
 $$
 \lim_{p\rightarrow\infty} \wh w_p = (\dot x, k) \ \hbox{in the sense of~$\wh d$}.
 $$
According to  Proposition\refer {ProofCFL1_Prop1}, if we set  
 $$
\cK_d(\dot x, k,Y) \eqdefa \lim_{p\rightarrow\infty} \cW(\wh w_p,Y)
 $$
 then the  definition of~$\wh f_\H$  on~$\wt \H^d$ and the Lebesgue dominated convergence theorem
imply that 
 $$\wh f_\H (\dot x, k) = \lim_{p\rightarrow\infty}\wh f_\H ( \wh w_p)
=\int_{\H^d} \ov\cK_d  (\dot x, k,Y)   f(Y,s) \,dY\,ds\,.$$
Now,  Lemma\refer {Wignerconvnormally}  gives
$$\displaylines{\quad
\cK_d (\dot x, k,Y) = \sum_{\ell_1,\ell_2}   \frac {(2i\eta)^{\ell_1}} {\ell_1! } \frac {(2y)^{\ell_2}} {\ell_2!}   
\lim_{p\rightarrow\infty}  \cH_{\ell_1,\ell_2} (\wh w_p)(\sgn\lam_p)^{\ell_1}\hfill\cr\hfill\quad\with
\cH_{\ell_1,\ell_2}(\wh w)\eqdefa|\lam|^{\frac{|\ell_1+\ell_2|}2}
\bigl(M^{\ell_1} H_m|\partial^{\ell_2} H_n\bigr)_{L^2(\R^d)}.}
 $$ 
 If $d=1$ then Corollary\refer {ProofCFL1_Coroll1} implies that 
   $$ 
\lim_{p\rightarrow\infty}  \cH_{\ell_1,\ell_2} (\wh w_p) =  F_{\ell_1,\ell_2} (k) \,\biggl(\frac{|\dot x|}4\biggr)^{\frac{\ell_1+\ell_2}2}
 $$
and, because $\sgn(\lambda_p)=\sgn \dot x$ for large enough $p,$  this guarantees \refeq {FourierL1basiceq2}
 and  Formula \eqref {definPhaseFlambda=0}.
 \smallbreak
Once again, as  in general dimension $d\geq1$ the term $\cH_{\ell_1,\ell_2}$ may 
be written as the product of $d$ terms involving only one-dimensional Hermite functions, 
the above formula still holds true (with the notation convention given in  
Proposition\refer{ProofCFL1_Prop1} of course). 

This concludes the proof of the first part of Theorem\refer  {FourierL1basic} and of Identity\refeq {FourierL1basiceq2}. 
\qed

 \begin{remark}\label{rk:K0}{\sl 
 Computing $\cK_d$ will be carried out later on, in Section\refer{proofFormulacK}. 
 For the time being, let us just underline that  the expression of $F_{\ell_1,\ell_2}(k)$ which appears  in\refeq{definPhaseFlambda=0} ensures that  $F_{0,0}(k)=\delta_{0,k}.$ We thus have 
\begin{equation}
\label {eq:K0} \cK_d(\dot x, k,0) = \cK_d(0, k,Y) =F_{0,0}(k)=\delta_{0,k}.
\end{equation}
Let us also notice that, denoting by~$\wh 0$ the point~$(0,0)$ of~$\wh\H^d_0$, we recover the following property:
\beq
\label {Fourier0=int}
\wh f_\H (\wh 0) =\int_{\H^d} f(w)\, dw.
\eeq
}
\end{remark}

\section {The case of functions that  do not depend on the vertical variable}
\label {FourierHorizontal}

The purpose of this section  is to prove Theorem\refer {Fourierhorizontal}. As
already pointed out in the introduction, a key issue is to  study  the
limit (in the sense of weak convergence of measures) of  functions which concentrate near the set~$\wh\H^d_0$. This is the aim of the following lemma.

\begin{lemma}
\label {convergesimplecouchH_0}
{\sl
Let $\wh\chi:\R\to\R$  be integrable, compactly supported and
with integral $1.$ 
Then for any   continuous
function $\theta$ from~$\wh\H^d$ to~$\C$  satisfying
\begin{equation}\label{eq:condtheta}
\sup_{(n,m,\lam)\in\wt\H^d}\bigl( 1+|\lam|( |n+m|+d) +|n-m|\bigr) ^{2d+1} | \theta (n,m,\lam)| <\infty,
\end{equation}
we have
$$
 \lim_{\ep\to0} \int_{\wh\H^d} \ep^{-1}\wh\chi(\ep^{-1}\lam)
\theta(n,m,\lam)\,d\wh w =  \langle \mu_{\wh\H^d_0},\theta\rangle
$$
where the measure in the right-hand side has been defined in \eqref{limitmeasureeq1}.
}
\end{lemma}
\begin{proof}
 Let us first prove the result if  the function $\theta$ 
 is supported in the closure of
 $$
 \cB_K\eqdefa\bigl\{(n,m,\lambda)\in \wt \H^d\,:\,
|\lam|(2|n|+d)\leq K\ \hbox{ and }\ |m-n|\leq K\bigr\}
$$
for some positive~$K$. Then  we have
\beno
\cI_\e &\eqdefa &  \int_{\wh\H^d} \ep^{-1}\wh\chi(\ep^{-1}\lam)
\theta(n,m,\lam)\,d\wh w= \sum_{|k|\leq K}\bigl(\cI_\e^-(k)+ \cI_\e^+(k)\bigr)\with\\
  \cI_\ep^\pm(k) & \eqdefa  &\int_{\R_{\pm}} \ep^{-1}\wh \chi(\ep^{-1}\lam) \biggl(\sum_{n\in\N^d}
\theta(n,n+k,\lam)\biggr)|\lambda|^dd\lambda.
\eeno
Above, we agreed  that~$\theta(n,n+k,\lam) =0$ whenever at least one component of~$n+k$ is negative.
Then the  idea is to use Riemann type  sums. More concretely, for all $n$ in~$\N^d$ and $\lambda$ in~$\R\setminus\{0\},$
let us define the family of cubes~$Q_{n,\lam} \eqdefa 2\lam n + 2\lam[0,1[^d$. It is obvious that  
\beq
\label {convergesimplecouchH_0demoeq0}
{\rm Vol} (Q_{n,\lam}) = (2|\lam|)^{d}\andf \sum_{n\in\N^d} {\bf 1}_{Q_{n,\lam}}=1\ \hbox{ on }\ (\R_{\sgn\lam})^d.
\eeq
{}From the volume property and the definition of $\cI_\e^+(k),$ we readily get 
$$
  \cI_\e^+(k) = 2^{-d}  \int_\R\int_{(\R_+)^d} \sum_{n\in\N^d}
 \ep^{-1}\wh \chi(\ep^{-1}\lam)
\theta(n,n+k,\lam) {\bf 1}_{Q_{n,\lam}}(\dot x)  \,d\dot x\,d\lambda.
$$
 Let us write that
$$
\longformule{
2^d\cI_\e^+(k)=\int_\R\int_{(\R_+)^d}\sum_{n\in\N^d}\e^{-1}\wh\chi(\e^{-1}\lam)\theta(\dot x,k) {\bf 1}_{Q_{n,\lam}}(\dot x)  \,d\dot x \,d\lam
}
{ {}
+ \int_\R\int_{(\R_+)^d}
 \ep^{-1}\wh \chi(\ep^{-1}\lam)  \sum_{n\in\N^d} \bigl(\theta(n,n+k,\lam)-\theta(\dot x,k)\bigr)
  {\bf 1}_{Q_{n,\lam}}(\dot x)  \,d\dot x\,d\lambda\,.
  }
$$
Using the second property of \eqref{convergesimplecouchH_0demoeq0},
the fact that~$\wh \chi$ is of integral~$1,$ and 
that the summation  may be restricted 
to those indices $n$ in~$\N^d$ such that~$|\lam n|\leq K$ 
(because $\theta$ is supported in~$\cB_K$), we end up with
$$
2^d\cI_\e^+(k)-\int_{(\R_+)^d}\!\theta(\dot x,k)\,d\dot x
=\int_\R\!\int_{(\R_+)^d}\!
 \ep^{-1}\wh \chi(\ep^{-1}\lam) \!  \sum_{|n\lam|\leq K} \bigl(\theta(n,n+k,\lam)-\theta(\dot x,k)\bigr)
  {\bf 1}_{Q_{n,\lam}}(\dot x)  \,d\dot x\,d\lambda.
$$
As~$\theta$ is uniformly continuous on~$\wh\H^d$ (being compactly supported), we have
$$
\forall \eta>0\,,\ \exists \e>0,\  |2\lam n-\dot x|+|\lam| <\e \Longrightarrow 
\bigl | \theta(n,n+k,\lam)-\theta(\dot x,k)\bigr | <\eta.
$$
One can thus conclude  that  for any $\eta>0,$ if $\e$ is small enough then we have
$$
\bigg|2^d\cI_\e^+(k)-\int_{(\R_+)^d}\theta(\dot x,k)\,d\dot x\biggr|
\leq  \eta \int_\R
 \ep^{-1}\wh \chi(\ep^{-1}\lam) \biggl( \sum_{|n\lam|\leq K}  \int_{(\R_+)^d} {\bf 1}_{Q_{n,\lam}}(\dot x)  \,d\dot x\biggr)
d\lambda.
$$
Using once again that the measure of $Q_{n,\lam}$ is $(2|\lam|)^d$ and 
noting  that the  set of indices $n$ in~$\N^d$ for which $|n\lam|\leq K$ is bounded by $C_d K^d|\lam|^{-d}$
for some constant $C_d$ depending only on $d,$ 
we conclude that for small enough $\e,$ we have
\beq
\label {convergesimplecouchH_0demoeq1}
\Bigl | \cI_\e^+(k) -2^{-d} \int_{(\R_+)^d}  
\theta(\dot x,k)\, d\dot x\Bigr | \leq C_d\eta K^d.
\eeq
Of course, handling $\cI_\e^-(k)$ is strictly similar.
Because the set of $k$ in~$\Z^d$  with $|k|\leq K$  is finite (and independent of $\e$), 
this proves the lemma in the case where $\theta$ is compactly supported. 
\medbreak
To handle the general case, one may  fix some cut-off function $\psi:\R_+\to\R_+$ with value $1$ on $[0,1]$ and 
supported in $[0,2],$   and, for all $K>0,$  decompose $\theta$
into $$\theta= \theta_K+\theta^K\with\theta_K(\wh w)\eqdefa \psi(K^{-1}(|\lam|(2|n|+d)+|m-n|))\theta(\wh w).$$ 
The first part of the proof applies to $\theta_K$  and for all  positive real number~$\eta$, one may thus find some $\ep_{K,\eta}$ so that 
\begin{equation}\label{eq:un}
\biggl|\int_{\wh\H^d}\ep^{-1}\wh\chi(\ep^{-1}\lam)\theta_K(\wh w)\,d\wh w-\langle \mu_{\wh\H_0^d},\theta_K\rangle\biggr|
\leq\eta\ \hbox{ for }\ \ep<\ep_{K,\eta}.
\end{equation}
To bound  the term corresponding to $\theta^K,$ we shall use the fact that Condition \eqref{eq:condtheta} ensures that there exists some constant $C$ so that 
\begin{equation}\label{eq:thetadotx}
\forall  (\dot x,k)\in\wh\H^d_0\,,\ |\theta(\dot x,k)|\leq C(1+|\dot x|+|k|)^{-2d-1}.
 \end{equation}
 Now, we have, denoting $\R^d_{\mp}\eqdefa (\R_-)^d\cup(\R_+)^d,$
 $$
 \int_{\wh\H^d_0} |\theta^K(\dot x,k)|\,d\mu_{\wh\H_0^d}(\dot x,k)\leq 
 2^{-d}\biggl(\sum_{|k|\geq K} \int_{\R^d_{\mp}} |\theta(\dot x,k)|\,d\dot x
 +\sum_{k\in\Z^d}\int_{|\dot x|\geq K}  |\theta(\dot x,k)|\,d\dot x\biggr)\cdotp
$$
In light of \eqref{eq:thetadotx} and making an obvious change of variables, we get
$$\begin{aligned}
\sum_{|k|\geq K} \int_{\R^d_{\mp}} |\theta(\dot x,k)|\,d\dot x&\leq C\sum_{|k|\geq K} \int_{\R_+^d}
(1+|\dot x|+|k|)^{-2d-1}\,d\dot x\\
&\leq C\sum_{|k|\geq K} (1+|k|)^{-d-1} \int_{\R^d_+} (1+|\dot y|)^{-2d-1}\,d\dot y
\leq C K^{-1}.\end{aligned}
$$
Likewise, 
$$\begin{aligned}
\sum_{k\in\Z^d}\int_{|\dot x|\geq K}  |\theta(\dot x,k)|\,d\dot x&\leq C\sum_{k\in\Z^d}\int_{|\dot x|\geq K}(1+|\dot x|+|k|)^{-2d-1}\,d\dot x\\
&\leq C \sum_{k\in\Z^d} \frac1{(1+|k|)^{d+1}}\int_{|\dot y|>K/(1+|k|)} \frac{d\dot y}{(1+|\dot y|)^{2d+1}}\\
&\leq C \sum_{k\in\Z^d}   \frac1{(1+|k|)^{d+1}} \frac1{(1+K/(1+|k|))^{d+1}}\\
&\leq C K^{-1}.\end{aligned}
$$
Therefore, if we take $K$ large enough then one may ensure that 
\begin{equation}\label{eq:deux}
\bigl|\langle \mu_{\wh\H_0^d},\theta^K\rangle\bigr|\leq\eta.
\end{equation}
Finally, 
$$\begin{aligned}
\biggl|\int_{\wh\H^d}\ep^{-1}\wh\chi(\ep^{-1}\lam)\theta^K(\wh w)\,d\wh w\biggr|
&\leq \cJ_K^1(\ep)+\cJ_K^2(\ep)\with\\
 \cJ_K^1(\ep)&\eqdefa\ep^{-1}\int_\R \sum_{|k|\geq K} \sum_{n\in\N^d} \wh\chi(\ep^{-1}\lam)|\theta(\wh w)|\,|\lam|^dd\lam
\andf\\ \cJ_K^2(\ep)&\eqdefa\e^{-1} \int_\R \sum_{k\in\Z^d} \sum_{|n\lam|\geq K} \wh\chi(\ep^{-1}\lam)|\theta(\wh w)|\,|\lam|^dd\lam.
\end{aligned}
$$
Because $\theta$ satisfies \eqref{eq:condtheta}, we have
$$
 \cJ_K^1(\ep)\leq C \sum_{|k|\geq K} \sum_{n\in\N^d} \int_\R\ep^{-1}\wh\chi(\ep^{-1}\lam) (1+|k|+|\lam n|)^{-2d-1}\,|\lam|^dd\lam.
 $$
 Clearly,  because the sum below has  $\cO(|k|/|\lam|)^d$ terms, we may write
  $$\int_\R \ep^{-1}\wh\chi(\ep^{-1}\lam)  \sum_{|n\lam|\leq|k|} (1+|k|+|\lam n|)^{-2d-1}\,|\lam|^dd\lam  
  \lesssim (1+|k|)^{-d-1} 
  $$
  and, similarly, because  $$\sum_{|n\lam|\geq|k|} |\lam|^d(1+|k|+|\lam n|)^{-2d-1}
  \lesssim \sum_{|n\lam|\geq|k|} |\lam|^d(1+|\lam n|)^{-2d-1}\lesssim (1+|k|)^{-d-1},$$ we get
   $$\int_\R \ep^{-1}\wh\chi(\ep^{-1}\lam)  \sum_{|n\lam|\geq|k|} (1+|k|+|\lam n|)^{-2d-1}\,|\lam|^dd\lam  
  \lesssim (1+|k|)^{-d-1}. 
  $$
  Therefore
  $$ \cJ_K^1(\ep)\lesssim K^{-1}.  $$
  Proving that $\cJ_K^2(\ep)\lesssim K^{-1}$ relies on similar arguments. Putting together with\refeq{eq:un}
  and\refeq{eq:deux}, it is now easy to conclude the proof of the lemma.  
\end{proof}
\begin{proof} [Proof of  Theorem\refer {Fourierhorizontal}] 
Let~$\chi$ in~$\cS(\R)$ have a compactly supported  Fourier transform, and value $1$ at $0$
(hence the integral of $\wh\chi$ is~$2\pi$).  Let~$\theta:\wh\H^d\to\C$ be continuous and compactly
supported,
and set
$$
\cI_\ep(g,\theta)\eqdefa\langle  \cF_\H(g\otimes\chi(\e\cdot)),
\theta\rangle.
$$
  By definition of the Fourier transform of~$L^1$ functions, one may write:
\beno
\cI_\e(g, \theta) & =  & \int_{\H^d\times \wh \H^d}e^{-is\lam} \chi(\e s) \ov\cW(\wh w,Y)  g(Y) \theta(\wh w) \,dY\,ds\,d\wh w\\
& = &\int_{\wh \H^d}   \frac 1 {\e} \wh \chi\Bigl(\frac  \lam {\e}\Bigr)   G(\wh w) \theta(\wh w) \,d\wh w\quad\with 
G(\wh w)  \eqdefa  \int_{T^\star \R^d} \ov\cW(\wh w,Y)  g(Y) \,dY.
\eeno
As the function~$g$ is integrable on~$T^\star\R^d$,  Proposition\refer  {ProofCFL1_Prop1} implies that the (numerical) product~$G\theta$ is a continuous compactly supported function on~$\wh\H^d$. 
 Lemma\refer   {convergesimplecouchH_0} applied to this function~$G\theta$ implies that 
 $$
  \lim_{\e\rightarrow0} \cI_\e(g,\theta) 
  = 2\pi \int_{\wh \H^d_0} \cG_\H g(\dot x,k) \theta(\dot x,k)d\mu_{\wh \H^d_0} (\dot x,k).
$$
This means  that the measure~$\cF_\H (g\otimes\chi(\e\cdot)) d\wh w$ converges weakly to~$2\pi(\cG_\H g) d\mu_{\wh \H^d_0} $ which is exactly Theorem\refer {Fourierhorizontal}.
\end{proof}


\section{Computing the kernel~$\cK$}
\label {proofFormulacK}

We have already seen  in Remark \ref{rk:K0}  that $\cK_d(0,k,Y)=\delta_{0,k}$ for all $Y$ in $T^\star\R^d,$
so let  us now prove the symmetry identities pointed out in the introduction. 
The first relation in \eqref{eq:Ksym} stems  from the  observation that for all $(n,m,\lam)$ in $\wt \H^d$ 
 and $Y$ in~$T^\star\R^d,$ we have
$$
\cW(m,n,\lam,-Y)=\overline{\cW(n,m,\lam,Y)}.
$$
Therefore, for any $(\dot x,k)$ in $\wh\H^d_0$
 passing to the limit $(n,m,\lam)\to(\dot x,k)$ yields
\begin{equation}\label{eq:Kconj}
\cK_d(\dot x,-k,-Y)=\overline{\cK_d(\dot x,k,Y)}.
\end{equation}
In order to establish the second symmetry relation for $\cK_d,$
it suffices to notice that 
 \begin{equation}\label{eq:symW}
\forall (n,m,\lam,Y) \in \wt \H^d\times T^\star \R^d\,,\   \cW(n,m,\lam,Y)=(-1)^{|n+m|}\cW(m,n,-\lam,Y).
\end{equation}
and to pass to the limit $(n,m,\lam)\to(\dot x,k).$
\smallbreak
The last relation in \eqref{eq:Ksym} just follows from passing to the  limit $(n,m,\lam)\to(\dot x,k)$ in
\begin{equation}\label{eq:WWW}
\cW(n,m,-\lam,Y)
=\ov\cW(n,m,\lambda,Y).
\end{equation}
Identity \eqref{eq:KLap} is a consequence of Relation\refeq {DeltaWignerHermite}. Indeed,  observe that  for any smooth 
function~$f:T^\star\R^d\to\C$, we have
$$
e^{-is\lam} \D_\H \bigl (e^{is\lam}  f(Y)\bigr) = \Delta_Y f(Y) + 4i\lam\sum_{j=1}^d \cT_j f(Y) -4\lam^2 |Y|^2f(Y)\with
\cT_j \eqdefa \eta_j\partial_{y_j} -y_j\partial_{\eta_j}.$$
Taking~$f(Y)= \cW(\wh w,Y),$ using\refeq {DeltaWignerHermite} and  having~$(n,m,\lam)$ tend to~$(\dot x,k)$ yields
\beq
\label  {computcKdemoeq1}
\D_Y \cK_d (\dot x,k,Y)  = -4 |\dot x|  \cK_d (\dot x,k,Y).
\eeq

Relation\refeq{eq:Kk} is a consequence of\refeq  {Fourierhorizontaldemoeq112} which implies in particular that
$$
|\lam| (n_j-m_j) \cW(\wh w,Y) =  i\lam \cT_j \cW (\wh w,Y).
$$
Passing to the limit when~$(n,m,\lam)$ tends to~$(\dot x,k)$ ensures
\beq
\label  {computcKdemoeq2}
ik _j\cK_d(\dot x,k,Y)  =  {\rm sgn}(\dot x) \cT_j \cK_d(\dot x,k,Y) 
\eeq
which is exactly\refeq  {eq:Kk}.
\smallbreak 
Proving Identity\refeq{Convollam=0} is bit more involved. To achieve it, let us fix some function $\al$ of~$\cS(\R)$   and two functions~$g_1$ and~$g_2$  of~$\cS(T^\star\R^d)$. By definition of convolution and  Fourier transform, we have 
$$
\longformule{
\cF_\H\bigl( (g_1\otimes \al)\star (g_2\otimes \al) \bigr) (\wh w)
}
{ {}
=\int_{\H^d\times \H^d} e^{-is\lam} \ov\cW(\wh w,Y) g_1(Y-Y') \al\bigl(s-s'-2\s(Y',Y)\bigr) g_2(Y') \al(s')\, dw \, dw'.
}
$$
Integrating first with respect to~$s$ and next with respect to  $s'$ yields 
$$
\cF_\H\bigl( (g_1\otimes \al)\star (g_2\otimes \al) \bigr) (\wh w)
=\wh \al^2(\lam) \int_{(T^\star\R^d)^2} e^{2i\lam \s(Y,Y')} \ov\cW(\wh w,Y) g_1(Y-Y') g_2(Y')  \,dY\,dY'.
$$
{}From the fact that~$\s$ is symplectic, we infer that 
\beq
\label {FormulaconvWgene}
\begin{aligned}
& \cF_\H\bigl( (g_1\otimes \al)\star (g_2\otimes \al) \bigr) (\wh w)\\
& \qquad\qquad\qquad{} =\wh \al^2(\lam) \int_{(T^\star\R^d)^2} e^{2i\lam \s(Y_1,Y_2)}  
\ov\cW(\wh w,Y_1+Y_2)g_1(Y_1) g_2(Y_2)\,dY_1\,dY_2.
\end{aligned}
\eeq
Of course, because both $g_1\otimes \al$ and $g_2\otimes \al$ are in $\cS(\H^d),$ we are guaranteed, thanks
to the convolution formula \refeq  {newFourierconvoleq1}, that
$$
\cF_\H\bigl( (g_1\otimes \al)\star (g_2\otimes \al) \bigr) (n,n+k,\lambda)=  G_{12}\with
G_{12}  \eqdefa  \bigl(\cF_\H(g_1\otimes \al)\cdot \cF_\H(g_2\otimes \al)\bigr) (n,n+k,\lam).
$$
Now, we have, setting $k'=n+k-\ell$ in the second line,
$$\begin{aligned}
G_{12} & =   \sum_{\ell\in\N^d} \cF_\H(g_1\otimes \al) (n,\ell,\lam)  \cF_\H(g_2\otimes \al) (\ell,n+k,\lam) \\
 &=  \wh \al^2(\lam) \int_{(T^\star\R^d)^2} \sum_{k'\leq n+k} 
\ov\cW(n,n+k-k',\lam,Y_1) \ov\cW(n+k-k',n+k,\lam,Y_2)\\
&\hspace{10cm}\times g_1(Y_1)g_2(Y_2)\, dY_1\,dY_2.
\end{aligned}
$$
Hence, reverting to Relation\refeq {FormulaconvWgene} and keeping in mind that
the above computations hold true for any functions $\alpha,$ $g_1$ and $g_2$ in the Schwartz class, 
one may conclude that
$$
e^{-2i\lam \s(Y_1,Y_2)}   \cW(n,n+k,\lam,Y_1+Y_2) 
=  \sum_{k'\in\Z^d}  \cW(n,n+k-k',\lam,Y_1) \cW(n+k-k',n+k,\lam,Y_2).
$$
Taking advantage of the decay of~$\cW$ with respect to the variable~$k$, (for~$Y_1$ and~$Y_2$ in a given compact subset of~$T^\star \R^d$ by virtue of\refeq  {Fourierhorizontaldemoeq112}), we can pass to the limit for~$2\lam n$ tending to~$\dot x$ and~$\lam$ tending to~$0$.  This gives
\begin{equation}\label{eq:Kd}
\cK_d( \dot x,k,Y_1+Y_2) = \sum_{k'\in \Z^d} \cK_d( \dot x,k-k',Y_1)\,\cK_d( \dot x,k',Y_2)
\end{equation}
which is the generalization of Formula\refeq {Convollam=0} in any dimension.

\medbreak In order to fully benefit from 
 Relations\refeq {eq:KLap},\refeq{eq:Kk} and\refeq {Convollam=0} so as to eventually compute $\cK,$
 it is wise to introduce the following function~$\wt \cK$ on $\R\times\TT\times T^\star\R,$
 where $\TT$ denotes the one-dimensional torus: 
\beq
\label {wtKdef}
\wt \cK(\dot x ,z,Y) \eqdefa  \sum_{k\in \ZZ} \cK(\dot x, k,Y) e^{ikz}.
\eeq
{}From Relation\refeq  {Fourierhorizontaldemoeq112} (after having~$(n,m,\lam)$ tend to~$(\dot x,k)$), we infer that if~$(\dot x,Y)$ lies in any given bounded set $\cB,$ then
\begin{equation}\label{eq:fastdecayK}
\forall N \in \N\,,\ \sup_{(\dot x,k,Y)\in\cB} (1+|k|)^N | \cK(\dot x,k,Y) |<\infty.
\end{equation}
 Thus the series\refeq {wtKdef}  defines  a function~$\wt\cK$ on $\R\times\TT\times T^\star\R.$

\medbreak
 {}Furthermore, from \refeq  {Convollam=0} we infer immediately  that 
\beq
\label {computcKdemoeq4}
\wt \cK(\dot x ,z,Y_1+Y_2)  =  \wt \cK(\dot x ,z,Y_1) \, \wt \cK(\dot x ,z,Y_2),
\eeq
and,  in light  of \eqref{eq:Kconj}, we discover that 
for any $(\dot x,z,Y)$  in $\R\times\TT\times T^\star\R,$
\begin{equation}
\wt\cK(\dot x,z,-Y)=\overline{\wt\cK(\dot x,z,Y)}.
\end{equation}
Combined with  \eqref{eq:K0}    and \eqref{computcKdemoeq4}
this implies that for any couple~$(\dot x,z)$ in $\R\times\TT,$ the function~$Y\mapsto  \wt \cK( \dot x ,z,Y)$  is a character of $\R^2.$
Identifying $T^\star\R$ with $\R^2,$ we thus conclude that  there exists a function $\Phi=(\Phi_y,\Phi_\eta)$  from~$\R\times \TT$ to~$\R^2$ such that 
$$
\wt\cK (\dot x,z,Y)= e^{iY\cdot \Phi(\dot x,z)}= e^{i(y\Phi_y(\dot x,z)+\eta\Phi_\eta (\dot x,z))}.
$$
Taking advantage of\refeq {computcKdemoeq1} which implies that~$\cK$ is a smooth function of~$Y$ and arguing as above,  we find out that for any multi-index $\alpha= (\alpha_1, \alpha_2)$ in $\N^2$  and any~$(\dot x,k,Y)$  in some bounded set~$\cB,$ we have
$$
\forall N \in \N\,,\ \sup_{(\dot x,k,Y)\in\cB} (1+|k|)^N |\partial^\alpha_{\dot x,Y} \cK(\dot x,k,Y) |<\infty.
$$
Therefore  invoking Relation\refeq {eq:Kk},  we deduce that  for any positive $\dot x$
$$
\partial_z  \wt \cK(\dot x,z,Y)= \eta\partial_y \wt \cK (\dot x,z,Y) - y \partial_\eta \wt \cK (\dot x,z,Y)
$$
which entails that $\partial_z\Phi(\dot x, z) =  R \Phi(\dot x, z)$
where~$R$ denotes the rotation of angle~$\pi/ 2$. Hence
$$
\Phi(\dot x, z) =  R(z)\wt \Phi(\dot x)
$$
where~$R(z)$ denotes the rotation of angle~$z$. Now, Relation\refeq {computcKdemoeq1} 
ensures that~$|\wt \Phi(\dot x) | = 2|\dot x|^{\frac 12},$ and thus there exists a function
$\phi$ from $\R$ to the unit circle of~$\R^2$ so that for positive~$\dot x$
\beq
\label {computcKdemoeq5}
\wt\cK(\dot x,z,Y) = e^{2i|\dot x|^{\frac 12} Y\cdot (R(z) \phi(\dot x))}.
\eeq
Let us finally establish  Identity\refeq  {Y2FouriercK}.
It relies on the study of the action of the Fourier transform on the \emph{weight function} $M^2$ defined by
$$
 (M^2f)(Y,s)\eqdefa|Y|^2f(Y,s).
$$
For any functions~$g$ in~$\cS(T^\star \R)$ and~$\psi:\R_+\times\Z\to\R,$ smooth and compactly supported in~$[r_0,\infty[\times \ZZ$  for some positive real number~$r_0$, let us define 
\beno
\Theta_\psi (\wh w) &\eqdefa & \psi\bigl(|\lam|(n+m+1),m-n\bigr) \andf\\
\cB(g,\psi)  & \eqdefa &    \int_{T^\star \R\times \wh \H_0^1} |Y|^2 \cK(\dot x,k,Y) g(Y) \psi(\dot x,k)\, dY d \mu_{\wh\H^1_0} (\dot x,k).
\eeno
Lemma\refer {convergesimplecouchH_0} implies that  if $\wh\chi:\R\to\R$ is  integrable, supported in $[-1,1]$ and with integral~$1$,~then
\beno
\cB(g,\psi)  & = & \lim_{\e\rightarrow 0} \cB_\e(g,\psi) \with \\
\cB_\e(g,\psi) & \eqdefa  &  \int_{T^\star \R} g(Y) \int_{\R} \sum_{(n,m)\in \N^2} |Y|^2 \cW(n,m,\lam,Y) 
\Theta_\psi (n,m,\lam)   \frac 1 \e  \wh \chi\Bigl(\frac \lam \e\Bigr) \,|\lam| \,d\lam\,dY.
\eeno
The following lemma gives a formula for~$ |Y|^2 \cW(\wh w,Y)$.
\begin{lemma}
 \label {Y2WignerHermite}
{\sl
For all~$\wh w$ in~$\wt\H^d$ and $Y$ in~$ T^\star\R^d,$ we have 
 $$|Y|^2\cW(\wh w,Y)  = -\wh\D \cW(\cdot ,Y) (\wh w) \with $$ \vspace{-8mm}
 \begin{multline} \label {decayWignerHermiteeq1}
 \wh \D \theta(\wh w) \eqdefa - \frac 1{2|\lam|} ( |n+m| +d) \theta(\wh w) 
\\[-1ex]+\frac 1 {2|\lam|} \sum_{j=1} ^d \Bigl\{ \sqrt {(n_j+1) (m_j+1)}\, \theta(\wh w^+_j) +\sqrt {n_jm_j}\, \theta(\wh w^-_j)\Bigr\}
 \end{multline}
 where~$\wh w^\pm_j \eqdefa (n\pm \d_j, m\pm\d_j, \lam)$. 
}
\end{lemma}
\begin{proof}
{}From the definition of~$\cW$ and  integrations by parts, we get  
\beno
|Y|^2\cW(\wh w,Y) & =  &\int_{\R^d} \Bigl(|y|^2-\frac 1 {4\lam^2} \D_z\Bigr) \bigl(e^{2i\lambda\langle \eta,z\rangle}\bigr)  H_{n,\lam} (y+z) H_{m,\lam} (-y+z)\, dz \\
&= & \int_{\R^d} e^{2i\lambda\langle \eta,z\rangle} |\lam|^{\frac d2} \cI(\wh w,y,z) \,dz \\\with
 \cI(\wh w,y,z)   &\eqdefa &  \Bigl(|y|^2-\frac 1 {4\lam^2} \D_z\Bigr)\bigl (H_{n} (|\lam|^{\frac 12} (y+z))  H_{m} (|\lam|^{\frac 12} (-y+z))\bigr).
\eeno
Using Leibniz formula, the chain rule and  $4|y|^2=|y+z|^2+|y-z|^2+2(y+z)\cdot(y-z),$ we get 
\beno
\cI(\wh w,y,z) & = &  -\frac 1 {4\lam^2} \bigl( (\D_z-\lam^2 |y+z|^2) H_{n} (|\lam|^{\frac 12} (y+z)) \bigr)  H_{m} (|\lam|^{\frac 12} (-y+z)) \\
&&{}
-\frac 1 {4\lam^2} \bigl( (\D_z-\lam^2 |y-z|^2) H_{m} (|\lam|^{\frac 12} (-y+z)) \bigr)  H_{n} (|\lam|^{\frac 12} (y+z)) \\
&&{}-\frac 1 {2|\lam|}\sum_{j=1}^d(\partial_j H_{n}) (|\lam|^{\frac 12} (y+z)) (\partial_j H_{m}) (|\lam|^{\frac 12} (-y+z))\\
&&{}-\frac 12 (z+y)\cdot(z-y) H_{n} (|\lam|^{\frac 12} (y+z))  H_{m} (|\lam|^{\frac 12} (-y+z)).
\eeno
Using\refeq{relationsHHermiteD}, we end  up with
\beno
\cI(\wh w,y,z) &= & \frac 1{2|\lam|} (|n+m|+d) H_{n} (|\lam|^{\frac 12} (y+z))  H_{m} (|\lam|^{\frac 12} (-y+z))
\\
&&\qquad \qquad{}
-\frac 1{2|\lam|} \sum_{j=1}^d\Bigl\{
(\partial_j H_{n}) (|\lam|^{\frac 12} (y+z)) (\partial_j H_{m}) (|\lam|^{\frac 12} (-y+z))\\
&&\qquad \qquad\qquad \qquad\qquad \qquad{}
+(M_j H_{n}) (|\lam|^{\frac 12} (y+z)) (M_j H_{m}) (|\lam|^{\frac 12} (-y+z))\Bigr\}\cdotp
\eeno
Then, taking advantage of\refeq{relationsHHermiteCAb}, 
 we get Identity\refeq  {decayWignerHermiteeq1}.
\end{proof}

\medbreak
Let us resume to  the proof of Identity\refeq {Y2FouriercK}.   Using the above lemma for $d=1$ and performing 
obvious changes of  variable in the sum give
\beno
\cB_\e(g,\psi) & =  & - \int_{T^\star \R} g(Y) \int_{\R} \sum_{(n,m)\in \N^2} (\wh\D \cW(\cdot,Y) ) (n,m,\lam)
\Theta_\psi (n,m,\lam)  \frac 1 \e  \wh \chi\Bigl(\frac \lam \e\Bigr)|\lam|\, d\lam\,dY\\
& = &  -\int_{T^\star \R} g(Y) \int_{\R} \sum_{(n,m)\in \N^2}  \cW(n,m,\lam, Y)
(\wh \D\Theta_\psi) (n,m,\lam)  \frac 1 \e  \wh \chi\Bigl(\frac \lam \e\Bigr) |\lam|\,d\lam\,dY\,.
\eeno
The key to proving the convergence of $\cB_\e$ for $\e\to0$ is the asymptotic
description of  the operator~$\wh\D$ when~$\lam$ tends to~$0,$ given in the following lemma: 
\begin{lemma}
\label {whDeltaoverlim0}
{\sl
Let~$\psi$ be  a smooth function compactly supported in~$[r_0,\infty[\times \ZZ$  for some positive real number~$r_0$. Then
$$
\wh\D \Theta_\psi (n,m,\lam)   \simH 1   \Theta_{L\psi} (n,m,\lam) \with
(L\psi) (\dot x,k)  \eqdefa   \dot x \psi'' (\dot x,k) +\psi'(\dot x,k)-\frac {k^2}{4\dot x} \psi(\dot x, k) 
$$
where the notation~$\Theta_1 \simH p \Theta_2$ means that for any positive integer~$N$, there is a constant~$C_{N, p}$ such that for all~$(n,m,\lam)$ in~$\N^2\times ]0,\infty[ $ satisfying 
$$ 
\lam (n+m)\geq \frac {r_0}  2\andf \lam\leq \lam_0/(1+|n-m|),
$$
 with   a sufficiently small positive real number  $\lam_0$ depending only on $r_0$, we have
$$
\bigl|\Theta_1(n,m,\lam) -\Theta_2(n,m, \lam)\bigr|\leq C_{N, p} \, \lam^p  \, \bigl(1+|\lam|(|n+m|+1)+|m-n|\bigr)^{-N}.
$$
}
\end{lemma}
\begin{proof}
By definition of the operator~$\wh \D$, and for~$\lam>0,$  we have, denoting $k\eqdefa m-n$
and~$y\eqdefa \lam(n+m),$
$$-2\lam^2 \wh\D \Theta_\psi (\wh w)  = (y\!+\!\lam) \psi (y\!+\!\lam,k) 
- \lam\sqrt {(n+1)(m+1)} \, \psi(y\!+\!3\lam,k)- \lam\sqrt {nm} \, \psi (y-\lam,k).$$
 Using that 
 $$
 \lam ^2 n m = \frac {\lam^2} 4 (n+m)^2 - \frac   {\lam^2} 4 (m-n)^2 = \frac {y^2} 4 -  \frac  {\lam^2} 4 k^2\,,
 $$
 we get that 
$$
\lam\sqrt {(n+1)(m+1)} =    \frac y 2 \sqrt {1+\frac {4\lam} y +\Bigl(\frac  {4- k^2} { y^2}\Bigr) \lam^2 } 
\simH 3  \frac y 2  +\lam  -\frac {k^2} {4y} {\lam^2}  \andf
\lam\sqrt {nm} 
\simH 3 \frac y 2  -\frac {k^2}  {4y}  {\lam^2}.
$$
Writing the Taylor expansion  for~$\psi$ gives (omitting the dependency with respect to $k$ for notational simplicity),
\beno
(y+\lam) \psi (y+\lam)  & \simH 3 & y\psi(y) +\bigl( \psi(y)+y\psi'(y) ) \lam + \Bigl( \psi'(y) +\frac y 2 \psi''(y)\Bigr)\lam^2 \,, \\
- \lam\sqrt {(n+1)(m+1)} \, \psi(y+3\lam) & \simH 3 & -\frac y 2 \psi(y) -\Bigl( \psi(y) +\frac 3 2  y\psi'(y)\Bigr) \lam  \\
&&\!\!\!\!\!\!\!\!{} 
- \Bigl(\frac 9 4  y\psi''(y)+ 3\psi'(y)  -  \frac {k^2}  {4y}\psi(y)\Bigr) \lam^2\andf\\
- \lam\sqrt {nm} \, \psi(y- \lam) & \simH 3 &  -\frac y 2 \psi(y)  +\frac 12  y\psi'(y)\lam 
- \Bigl( \frac y {4} \psi''(y) - \frac {k^2}  {4y} \psi(y)\Bigr)  {\lam^2}.
\eeno 
By summation of these three identities, we get
$$
- 2 \lam^2 \wh\D \Theta_\psi (\wh w)  \simH 3  -  \Bigl( 2y\psi''(y) +2\psi'(y) - \frac {k^2} {2y} \psi(y)\Bigr) \lam^2,
$$
whence  the lemma.
\end{proof}

\medbreak
From the above  lemma,  it is easy to complete the proof of  Identity\refeq {Y2FouriercK}. 
Indeed, we get 
$$
\displaylines{
\cB_\e(g,\psi) =  - \int_{T^\star \R} g(Y)\!\!
 \int_{\R} \!\sum_{(n,m)\in \N^2} \!\! \cW(n,m,\lam, Y)
L\psi\bigl(|\lam|(n+m), m-n\bigr)  \frac 1 \e  \wh \chi\Bigl(\frac \lam \e\Bigr)|\lam| \, d\lam\,dY
 \hfill\cr\hfill+ \cR_\e(g,\psi),}
$$
where the remainder $\cR_\ep$ is such that for all $N\in\N$ there exists $C_N$ so that
$$\bigl|\cR_\e(g,\psi)\bigr| \leq C_N \|g\|_{L^1(T^\star\R^d)} \sum_{(n,m)\in \N^2} \int_{\R} |\lam|  \bigl(1+|\lam|(|n+m|+1)+|m-n|\bigr)^{-N} \frac 1 \e  |\wh \chi| \Bigl(\frac \lam \e\Bigr) |\lam| d\lam \,.
$$
Taking~$N$ large enough, we find out that
$$
 \sum_{(n,m)\in \N^2} \int_{\R} |\lam|  \bigl(1+|\lam|(|n+m|+1)+|m-n|\bigr)^{-N} \frac 1 \e  |\wh \chi| \Bigl(\frac \lam \e\Bigr) |\lam| \,d\lam
 \leq C_N\int_\R \frac {|\lam|}  \e  |\wh \chi| \Bigl(\frac \lam \e\Bigr) d\lam\leq C_N \e.
$$
Then Lemma\refer  {convergesimplecouchH_0} ensures 
$$
\cB(g,\psi)   = - \int_{T^\star \R\times \wh \H_0^1} \cK(\dot x,k,Y) g(Y)
\Bigl(\dot x \psi'' (\dot x,k) +\psi'(\dot x,k)-\frac {k^2}{4\dot x} \psi(\dot x, k)\Bigr) dY \,d \mu_{\wh\H^1_0} (\dot x,k).
$$
Integration by parts yields 
$$
\cB(g,\psi)  =  \int_{T^\star \R\times \wh \H_0^1} g(Y)
\Bigl(\frac {k^2}{4\dot x} \cK(\dot x, k,Y)-\partial_{\dot x}\cK(\dot x,k,Y)-\dot x \partial_{\dot x} ^2 \cK (\dot x,k,Y)\Bigr)  \psi(\dot x, k) \,dY \,d \mu_{\wh\H^1_0} (\dot x,k).
$$
Using the fact that the above equality holds true for all $g$ in $\cS(T^\star\R)$ and for functions~$\psi$
smooth and compactly supported in~$[r_0,\infty[\times\Z$ for some $r_0>0,$ and combining with
a density argument, one can  conclude to Identity\refeq {Y2FouriercK}  for all positive $\dot x$ and~$k$ in~$\Z.$
\medbreak
In order to complete the proof of\refeq{FourierL1basiceq2b}, let us translate\refeq {Y2FouriercK} 
in terms of $\wt\cK.$ We have
$$
\frac1{4\dot x}\partial_z^2\wt\cK +\partial_{\dot x}(\dot x\partial_{\dot x}\wt\cK)+|Y|^2\wt\cK=0.
$$
Now,  plugging the ansatz\refeq{computcKdemoeq5} into the above relation yields for any positive $\dot x$, any~$k$ in~$\Z$ and any~$Y$ in~$ T^\star\R,$
$$\displaylines{\quad
|Y|^2=
 \bigl(Y\cdot (R'(z)\phi(\dot x))\bigr)^2
+\bigl(Y\cdot(R(z)\phi(\dot x))+2\dot x Y\cdot(R(z)\phi'(\dot x))\bigr)^2
\hfill\cr\hfill
-4i\sqrt{\dot x} Y\cdot(R(z)\phi'(\dot x))-2i\dot x^{3/2}Y\cdot(R(z)\phi''(\dot x)).\quad}
$$
Taking the imaginary part implies that $\phi$ satisfies  
$$\dot x\phi''(\dot x)+2\phi'(\dot x)=0\quad\hbox{for }\ \dot x>0.
$$
Now, as $\phi$ is valued in the unit circle, this implies that $\phi$ is a constant. 
Therefore there exists some number $z_0$ in~$(-\pi,\pi]$ so that for any positive~$\dot x$, any~$z$ in~$\R$ and any~$Y$ in~$T^\star\R,$ we have
$$
\wt\cK(\dot x,z,Y) = e^{2i|\dot x|^{\frac 12}(y\cos(z+z_0)+\eta\sin(z+z_0))}.
$$
Inverse Fourier theorem for periodic functions implies that 
$$
\cK(\dot x,k,Y)=\frac1{2\pi}\int_{-\pi}^\pi  e^{2i|\dot x|^{\frac 12}(y\cos(z+z_0)+\eta\sin(z+z_0))} e^{-ikz}\,dz.
$$
In order to compute the value of $z_0,$ one may take advantage of the symmetry relations in\refeq{eq:Ksym} that imply
\begin{equation}\label{eq:ksym2}
\cK(\dot x,-k,y,-\eta)=(-1)^k\cK(\dot x,k,y,\eta).
\end{equation}
Now, the above formula for $\cK$ and an obvious change of variable give
$$
\begin{aligned}
2\pi\cK(\dot x,-k,y,-\eta)&= \int_{-\pi}^\pi e^{ikz}  e^{2i|\dot x|^{\frac 12}(y\cos(z+z_0)-\eta\sin(z+z_0))}\,dz\\
&= \int_{-\pi}^\pi e^{ik(\pi-z)}  e^{2i|\dot x|^{\frac 12}(y\cos(\pi-z+z_0)-\eta\sin(\pi-z+z_0))}\,dz\\
&=(-1)^k \int_{-\pi}^\pi e^{-ikz}  e^{-2i|\dot x|^{\frac 12}(y\cos(z-z_0)+\eta\sin(z-z_0))}\,dz.
\end{aligned}
$$
Hence \eqref{eq:ksym2} is fulfilled for all positive~$\dot x$,~$k$ in~$\Z$ and~$(y,\eta)$ in~$ T^\star\R$ if and only if 
$$
\forall  z\in(-\pi,\pi)\,,\  \cos(z+z_0)=-\cos(z-z_0)\quad\hbox{and} \quad\sin(z+z_0)=-\sin(z-z_0)
$$
which is equivalent to $z_0\equiv \frac\pi2 [\pi].$ 
Hence there exists $\e\in\{-1,1\}$ so that
$$
\wt\cK(\dot x,z,Y) = e^{2i\e\sqrt{\dot x}(y\sin z-\eta\cos z)}.
$$
To determine the value of $\ep,$  one may use the fact that for all  positive~$\dot x$ and~$\eta$ in~$\R,$ the above formula implies that 
$$
\sum_{k\in\Z} \cK(\dot x,k,(0,\eta))=\wt\cK(\dot x,0,(0,\eta))=e^{-2i\ep\sqrt{\dot x}\,\eta}
=\cos(2\sqrt{\dot x}\,\eta)-i\ep\sin(2\sqrt{\dot x}\,\eta).
$$
Now, from the expansion of $\cK$ given in\refeq{definPhaseFlambda=0}, we infer that for all $\eta\in\R$ and $\dot x>0,$
$$
\wt\cK(\dot x,0,(0,\eta))=\sum_{\ell_1\in\N}\sum_{|k|\leq\ell_1} \frac{i^{\ell_1}}{\ell_1!} F_{\ell_1,0}(k) \eta^{\ell_1}
 \dot x^{\frac{\ell_1}2}.
$$
Note that the imaginary part of the term corresponding to $\ell_1$ is positive (indeed $F_{1,0}(k)$ is positive),
which implies that $\ep=-1$.
This completes the proof of  Identity\refeq {FourierL1basiceq2b} in the case where $\dot x$ is non negative.
The negative case just follows from\refeq{eq:Ksym}. Thus the whole Theorem\refer {FourierL1basic} is proved.


\section {Some properties of operator~$\cG_\H$}
\label {StutyofcGH}

We end this paper with a short presentation of  basics properties of the transformation~$\cG_\H,$ 
that highlight some  analogy (but also some difference)  with  the classical Fourier transform on~$T^\star\R^d$. 
The main result of this section reads as follows.
\begin{theorem}
\label {FourierhorizontalMore}
{\sl
The operator $\cG_\H$ maps continuously~$L^1(T^\star\R^d)$ to the space~$\cC_0(\wh\H^d_0)$ 
of continuous functions on $\wh\H^d_0$ going to $0$ at infinity and, for any  couple 
$(f,g)$ of functions in~$L^1(T^\star\R^d),$ we have the convolution identity:
\begin{equation}\label{eq:convGH}
\cG_\H(f\star g)(\dot x,k)=\sum_{k'\in\Z^d} \cG_\H f(\dot x,k-k')\,\cG_\H(\dot x,k')
\quad\hbox{for all }\ (\dot x,k)\in\wh\H^d_0.
\end{equation}
Moreover, for any $g$  in~$\cS(T^\star\R^d),$  we have the following inversion formula:
$$
g(Y) = \Bigl(\frac 2 \pi \Bigr) ^d \int_{\wh \H^d_0} \cK_d  (\dot x,k,Y) \cG_\H g(\dot x,k) \,d\mu_{\wh \H^d_0} (\dot x,k).
$$
Finally,  the following Fourier-Plancherel   identity holds true:
$$
\forall g\in \cS(T^\star\R^d)\,,\ \|g\|_{L^2(T^\star\R^d)} ^2 =  \Bigl(\frac 2 \pi \Bigr) ^d  \|\cG_\H g\|_{L^2(\wh\H_0^d)}^2.
$$
}
\end{theorem}

\begin{proof}
The first property stems from the fact that, because~$|\cK_d|\leq1,$ we have
$$
 \|\cG_\H g\|_{L^\infty(\wh\H_0^d)} \leq  \|g\|_{L^1(T^\star\R^d)}.
 $$
  Furthermore, as  the kernel $\cK_d$ is continuous with respect to~$(\dot x,k),$ 
  we get from   the explicit expression of $\cG_\H$ that the range of $L^1(T^\star\R^d)$ by $\cG_\H$ 
  is included in the set of continuous functions  on~$\wh\H^d_0.$ 
  
  Proving  that $(\cG_\H g)(\dot x,k )$ tends  to $0$ when $(\dot x,k )$ goes to infinity
  is based on the regularity and decay properties of the kernel~$\cK_d.$ More specifically, 
  Identity\refeq {eq:KLap} implies that 
$$
\forall p \in \N\,,\  4^p|\dot x| ^p\cK_d(\dot x,k,Y) = \bigl((-\D_{Y} )^p \cK_d\bigr)(\dot x,k,Y),
$$
while Relation\refeq  {eq:Kk} gives for all multi-index $\alpha$ in $\N^d,$
$$
(ik \sgn \dot x)^\al \cK_d (\dot x,k,Y) = ( \cT^\al \cK_d) (\dot x,k,Y)\bigr)
\with  \cT^\al \eqdefa \prod_{j=1}^d ( \eta_j\partial_{y_j} -y_j\partial_{\eta_j}
)^{\al_j} . 
$$
Hence, if $g\in\cS(T^\star\R^d)$ then performing suitable integration by parts
in the integral defining~$\cG_\H g$ yields
$$
4^p|\dot x|^p (\cG_\H g)(\dot x,k )=   \cG_\H ((-\D_Y)^pg)(\dot x,k) \andf 
 (-i  k \sgn \dot x )^\al (\cG_\H g)(\dot x,k )=  (\cG_\H  \cT^\al g)(\dot x,k).
$$
This implies that, for any positive integer~$p$, a constant~$C_p$ and an integer~$N_p$ exist such that 
\ben
\label {decaydotx1}
( 1+|\dot x|+|k|) ^p|\cG_\H(g) (\dot x,k)| \leq C_p \, \|g\|_{N_p,\cS(T^\star\R^d)}.
\een
This proves that $(\cG_\H g)(\dot x,k )$ tends to~$0$ when~$(\dot x,k )$ goes to infinity
for any $g$ in~$\cS(T^\star\R^d).$ Now, because~$L^1(T^\star\R^d)$ is 
dense in~$\cS(T^\star\R^d)$ and~$\cG_\H$ is continuous from~$L^1(T^\star\R^d)$ to
the set~$\cC_b(\wh\H^d_0)$ of bounded continuous functions on $\wh\H^d_0$, 
one can conclude that the range of~$L^1(T^\star \R^d)$ by~$\cG_\H$ is included in~$\cC_0(\wh \H^d_0)$.
    \medbreak
     In order to establish\refeq{eq:convGH}, 
     it suffices to see that, by virtue of  the definition of $\cG_\H,$ 
     of Identity\refeq{eq:Kd} and of Fubini theorem (here
  the decay inequality\refeq{eq:fastdecayK} comes into play), one may write that for any couple $(f,g)$ of integrable functions on $T^\star\R^d,$ we have
  $$
  \begin{aligned}
  \cG_\H(f\star g)(\dot x,k)&=\int_{(T^\star\R^d)^2}\ov\cK_d(\dot x,k,Y)\,f(Y-Y') \, g(Y')\,dYdY'\\
  &=\sum_{k'\in\Z^d}\int_{(T^\star\R^d)^2}\ov\cK_d(\dot x,k',Y') g(Y')\:\ov\cK_d(\dot x,k-k',Y-Y')
  f(Y-Y')\,dYdY'.
  \end{aligned}
  $$
    Then performing an obvious change of variable, and using again Fubini  theorem and the definition of $\cG_\H$ gives\refeq{eq:convGH}. 
    \medbreak
  In order to prove the \emph{inversion Fourier formula for $\cG_\H$}, let us consider~$g$ in~$\cS(T^\star\R^d)$ and~$\chi$ in~$\cS(\R)$ with value~$1$ near~$0$. For any sequence~$\suite \e p \N$ of positive real numbers which tends to~$0$, we have according to the inverse Fourier formula~\eqref{MappingofPHdemoeq1},
 \beno
g(Y) \chi(\e_p s)  & =  & \frac {2^{d-1}}  {\pi^{d+1} }   \int_{\wt \H^d} 
e^{is\lam} \cW(\wh w, Y)\cF_\H (g\otimes \chi(\e_p\cdot)) (\wh w) \, d\wh w\\
& = &  \frac {2^{d-1}}  {\pi^{d+1} }   \int_{\wt \H^d} e^{is\lam} \cW(\wh w,  Y)
\Bigl( \int_{T^\star \R^d} \ov\cW (\wh w,Y') g(Y') dY'\Bigr) \frac 1 {\e_p} \wh \chi \Bigl(\frac \lam {\e_p}\Bigr) \, d\wh w.
\eeno
{}From the definition of $\D_\H$ in\refeq{defLaplace}, we gather that for any integer $p$ and positive real number~$\e$, there exist some  function
$f_\e^p$ on $\H^d,$ and constant  $C_p$ (depending only on  $p$) so that
\begin{equation}
\label {use1}
(-\D_\H)^p \chi(\e s) g(Y)=  \chi(\e s) (-\D_Y)^p g(Y) + \e f^p_\e(Y,s)\with  \|f^p_\e(\cdot,s)\|_{L^1(T^\star\R^d)}\leq C_p.
\end{equation} 
Therefore, having $\ep$ tend to $0,$ we  deduce that  
$$
 |\lam|^p (2|m|+d) ^p \Big| \int_{T^\star \R^d} \ov\cW (\wh w,Y) g(Y) \,dY \Big| \leq C_p    \int_{T^\star \R^d} \Big|(-\D_Y)^pg(Y)\Big| dY \, .
 $$
Along the same lines, taking advantage of 
the \emph{right-invariant} vector fields defined in\refeq{eq:rightinv}, we  get for any integer $p$ 
$$ |\lam|^p (2|n|+d) ^p \Big| \int_{T^\star \R^d} 
\ov\cW (\wh w,Y) g(Y) dY \Big| \leq C_p    \int_{T^\star \R^d} \Big|(-\D_Y)^pg(Y)\Big| dY \, .$$
Identity\refeq {Fourierhorizontaldemoeq112}  together with integrations by parts implies that  for  any multi-index $\al$
$$ 
\bigl(-i\,\sgn\lam\bigr)^{|\al|} \prod_{j=1}^d (n_j-m_j)^{\al_j}   \int_{T^\star \R^d} \ov\cW (\wh w,Y) g(Y) \,dY =  \int_{T^\star \R^d}  \ov\cW (\wh w,Y) \cT^\al g(Y) \,dY \, .
$$
We deduce that the function
$$
\wh w\longmapsto \cW(\wh w, Y)
\Bigl( \int_{T^\star \R^d} \ov\cW (\wh w,Y') g(Y')\, dY'\Bigr)
$$
satisfies the hypothesis of Lemma\refer {convergesimplecouchH_0}.
Thus combining with  Proposition\refer  {ProofCFL1_Prop1}   gives
$$\begin{aligned}
g(Y) & =   \frac {2^{d-1}}  {\pi^{d+1} }  \:2\pi 
 \int_{\wh \H_0^d}   \cK_d(\dot x,k,Y)
\Bigl( \int_{T^\star \R^d} \ov\cK_d (\dot x,k,Y') g(Y') dY'\Bigr)  d\mu_{\wh \H_0^d}(\dot x,k) \\
& =  \Bigl(\frac2\pi\Bigr)^{d}  \int_{\wh \H_0^d}   \cK_d(\dot x,k, Y)
\,\cG_\H g (\dot x, k) \, d\mu_{\wh \H_0^d}(\dot x,k),
\end{aligned}
$$
which completes the proof of the inversion formula.
\medbreak
Of course, as in the classical Fourier theory, having  an inversion formula implies a Fourier-Plancherel type relation. Indeed we have
 for any function~$g$ in~$\cS(T^\star\R^d),$  using Fubini theorem,
\beno
\int_{T^\star\R^d} g(Y)\overline g(Y) \,dY & =   &  \Bigl(\frac2\pi\Bigr)^{d} \int_{T^\star\R^d} \biggl( \int_{\wh \H_0^d}   \cK_d(\dot x,k, Y)
\,\cG_\H g (\dot x, k) \, d\mu_{\wh \H_0^d}(\dot x,k)\biggr) \overline g(Y) \,dY\\
& =   &  \Bigl(\frac2\pi\Bigr)^{d}  \int_{\wh \H_0^d}  \cG_\H g (\dot x, k)\overline {\biggl(  \int_{T^\star\R^d}  \ov\cK_d(\dot x,k,Y)
g(Y) \,dY\biggr) } \, d\mu_{\wh \H_0^d} (\dot x,k)\\
& = &  \Bigl(\frac2\pi\Bigr)^{d}  \int_{\wh \H_0^d}  \cG_\H g (\dot x, k)\,\overline {\cG_\H g (\dot x, k) } \, d\mu_{\wh \H_0^d} (\dot x,k).
\eeno
The whole Theorem\refer {FourierhorizontalMore} is proved.
\end{proof}


\appendix 
\section{Useful tools and results}
\label {FourierHbasic}


Let us first recall standard properties of Hermite functions 
that have been used repeatedly in the paper, when establishing identities pertaining 
to the function $\cW.$ 

In addition to the creation operator $C_j \eqdefa -\partial_j+M_j$ already defined in the introduction, 
 it is convenient to introduce the following  \emph{annihilation} operator: 
\beq
\label {definCreaAnnhil}
A_j \eqdefa \partial_j+M_j.
\eeq
It is very classical (see e.g. \cite{huet,O}) that 
\beq
\label {relationsHHermiteCA}
A_j H_{n} = \sqrt {2n_j}\,  H_{n-\d_j} \andf C_j H_n = \sqrt {2n_j+2}\, H_{n+\d_j},
\eeq
As Relations\refeq {definCreaAnnhil} imply that
\begin{equation}\label{Mjdj}
2M_j =C_j+A_j \andf 2\partial_j =A_j-C_j,
\end{equation}
 we discover that
\beq
\label {relationsHHermiteCAb}
\begin{array} {rcl}
M_j H_{n}  &= & \ds  \frac 12 \bigl(  \sqrt {2n_j} \,H_{n-\d_j}+\sqrt {2n_j+2}\,  H_{n+\d_j}\bigr) \andf \\ &&\\
\partial_j H_n &=  & \ds \frac 12\bigl (\sqrt {2n_j}\, H_{n-\d_j}- \sqrt {2n_j+2}\, H_{n+\d_j}\bigr).
\end{array}
\eeq
Note also  that 
\beq
\label {relationsCHA}
C_j A_j +\Id =-\partial_j^2+M_j^2 \andf [C_j,A_j] = -2\Id,
\eeq
and thus 
$$
\Delta_{\rm osc}^1=\sum_{j=1}^d C_jA_j + d\Id.
$$
Finally, we have 
\beq
\label {relationsCommuteCH}
[ -\partial_j^2+M_j^2, C_j] = 2C_j. 
\eeq


Let us next  prove Relation\refeq  {LaplacianHWignerj} and Lemma\refer {decaylambdan}. To this end, we write that by definition of $\cX_j$ and of $\cW,$ we have
\beno
\cX_j \bigl(e^{is\lam} \cW (\wh w,  Y) \bigr) & = & \int_{\R^d}  \cX_j \bigl(e^{is\lam +2i\lam \langle \eta,z\rangle} H_{n,\lam}(y+z)
H_{m,\lam} (-y+z)\bigr) dz\\
& = & \int_{\R^d}  e^{is\lam +2i\lam \langle \eta,z\rangle }
\bigl(2i \lam \eta_j +\partial_{y_j}\bigr) \bigl( H_{n,\lam}(y+z)
H_{m,\lam} (-y+z)\bigr) dz.
\eeno
As~$2i \lam \eta_j e^{2i\lam \langle \eta,z\rangle } = \partial_{z_j}(e^{2i\lam \langle \eta,z\rangle })$, an integration by parts implies that 
\beq
\label {MappingofPHdemoeq8}
\cX_j \bigl(e^{is\lam} \cW (\wh w, Y) \bigr) =  \int_{\R^d}  e^{is\lam + 2i\lam \langle \eta,z\rangle }
\bigl( \partial_{y_j}-\partial_{z_j}\bigr) \bigl( H_{n,\lam}(y+z)
H_{m,\lam} (-y+z)\bigr) dz.
\eeq
The action of~$\Xi_j$ is simply described by
\beno
\Xi_j \bigl(e^{is\lam} \cW (\wh w,Y) \bigr)  & = & \int_{\R^d}  \Xi_j \bigl(e^{is\lam + 2i\lam \langle \eta,z\rangle}\bigr) H_{n,\lam}(y+z)
H_{m,\lam} (-y+z)\, dz\\
& = & \int_{\R^d}  e^{is\lam + 2i\lam \langle \eta,z\rangle}
2i\lam(z_j -y_j)  H_{n,\lam}(y+z)
H_{m,\lam} (-y+z)\, dz.
\eeno
Together with\refeq  {MappingofPHdemoeq8}, this gives
$$
(\cX_j^2+\Xi_j^2)  \bigl(e^{is\lam} \cW(\wh w,Y)\bigr)  =4\int _{\R^d}  e^{is\lam + 2i\lam \langle \eta,z\rangle} 
H_{n,\lam}(y+z)
\bigl(( -\partial_j^2+\lam^2M_j^2\bigr)  H_{m,\lam}) (-y+z)\, dz.
$$
Putting together  with\refeq {relationsHHermiteD},   we get Formula\refeq  {LaplacianHWignerj} and thus\refeq {DeltaWignerHermite}.
\medbreak
Now, if $f$ belongs to the Schwartz space, then combining  \refeq {DeltaWignerHermite} with integrations by parts 
yields\refeq{FourierdiagDeltaHfond}. Indeed, we have
$$
\begin{aligned}
-4|\lam|(2|m|+d) \wh f_\H (n,m,\lam) & = -4|\lam|(2|m|+d)\int_{\H^d} e^{-is\lam} \ov\cW(\wh w,Y) f(Y,s) \,ds \\
  & =   \int_{\H^d}  \Delta_\H \bigl (e^{-is\lam} \ov\cW(\wh w,Y) \bigr) f(Y,s) \,ds \\
& =  (\cF_\H\D_\H f) (n,m,\lam).
\end{aligned}$$
This gives, after iteration
\beq
\label {decaylambdandemoeq1}
4^p|\lam|^p(2|m|+d)^p \bigl| \wh f_\H (n,m,\lam) \bigr | \leq \|\D_\H^p f\|_{L^1(\H^d)},
\eeq
which is one part of the decay inequality\refeq{eq:decay}.
\medbreak
To complete the proof of  Lemma\refer  {decaylambdan}, 
it suffices to exhibit suitable decay properties with respect to~$k$. 
To this end, we introduce the \emph{right-invariant}  vector fields
 $\wt\cX_j$ and $\wt\Xi_j$  defined  by 
 \begin{equation}\label{eq:rightinv}
 \wt\cX_j\eqdefa\partial_{y_j}-2\eta_j\partial_s\andf 
 \wt\Xi_j\eqdefa\partial_{\eta_j}+2y_j\partial_s\with j\in \{1,\cdots,d\}\,.
 \end{equation}
   Then arguing as above, we readily get
  $$
  \longformule{
  4|\lam| (2n_j+1) e^{is\lam} \cW(\wh w,Y) = -e^{is\lam} ( \partial_{y_j}^2+ \partial_{\eta_j}^2) \cW(\wh w,Y)
}
{ {}
+4i\lam e^{is\lam} (\eta_j \partial_{y_j} -y_j\partial_{\eta_j } )\cW(\wh w,Y)
-4  e^{is\lam} \lam^2 (y_j^2+\eta_j^2)\cW(\wh w,Y).
}
$$
As 
$$
 -\wt \cX_j^2+\cX_j^2 -\wt \Xi_j^2+\Xi_j^2= 8\partial_s\cT_j \, ,
 $$
 we get by difference with\refeq  {LaplacianHWignerj}, 
\beq
\label  {Fourierhorizontaldemoeq112}
|\lam| (n_j-m_j) e^{is\lam} \cW(\wh w,Y) =  
\partial_s \cT_j \bigl( e^{is\lam} \cW\bigr)  (\wh w,Y)= i \lam e^{is\lam} \cT_j \cW  (\wh w,Y).
\eeq
 After an obvious iteration this gives
\beq
\label  {Fourierhorizontaldemoeq1122}
\prod_{j=1}^d (n_j-m_j)^{\al_j}   \cW(\wh w,Y)=  ( i\,{\rm sgn } (\lam))^{|\al|}  
\cT^\al \cW(\wh w,Y),
\eeq
whence, mimicking the proof of\refeq{decaylambdandemoeq1}, 
\beq
\label {decaylambdandemoeq2}
|n-m|^p |\wh f_\H (n,m,\lam)| \leq  \sup_{|\al|=p} \|\cT^\al f\|_{L^1(\H^d)}\quad\hbox{for all }\ f\ \hbox{ in }\ \cS(\H^d).
\eeq
This completes the proof  of Lemma\refer {decaylambdan}. 
\medbreak
Let us finally prove the convolution identity\refeq{newFourierconvoleq1}.
It is just based on the fact that for all $(n,m,\lambda)$ in $\wt\H^d$ and any integrable function $f$ on $\H^d,$ we have
$$
 \sum_{m\in \N^d} |\wh f_\H(n,m,\lam)|^2\leq \|f\|^2_{L^1(\H^d)} \andf
 \sum_{n\in \N^d} |\wh f_\H(n,m,\lam)|^2\leq \|f\|^2_{L^1(\H^d)}.
$$
Indeed, 
 if~$A$ and~$B$ are two bounded operators on a separable Hilbert space~$\cH$ endowed with 
 an orthonormal basis~$\suite e n {\N^d}$ then, denoting 
 $$A(n,m) \eqdefa (Ae_m|e_n) \andf B(n,m) \eqdefa (Be_m|e_n),$$
one may write 
\beq
\label {newFourierconvoldemoeq1}
\sum_{\ell \in \N^d} |A (\ell, m)|^2=\|Ae_m\|_{\cH}^2\leq \|A\|_{\cL(\cH)}^2 \quad \hbox {and} \quad 
\sum_{\ell \in \N^d} |A(n,\ell)|^2=\|A^*e_n\|_{\cH}^2\leq \|A\|_{\cL(\cH)}^2.
\eeq
 Therefore, from  Inequality\refeq {L1LinftyFourierbasic}
 and Definition\refer{definFouriercoeffH},
  we readily infer  that 
$$\begin{aligned}
\bigl(\cF^\H(f)(\lam)\circ \cF^\H(g) (\lam) H_{m,\lam} \bigl| H_{n,\lam} \bigr)_{L^2(\R^d)}
& = \!\!\lim_{(N,N')\rightarrow(\infty,\infty)}\!\!\sum_{\substack {|\ell|\leq N\\ |\ell'|\leq N'}}\!\!
\wh f_\H(\ell',\ell,\lam)\wh g_\H(\ell,m,\lam) (H_{\ell',\lam}|H_{n,\lam})\\
& =  \!\lim_{N\rightarrow \infty}\sum_{|\ell|\leq N} \wh f_\H(n,\ell,\lam)\,\wh g_\H(\ell,m,\lam)
\\ & = \sum_{\ell\in\N^d} \wh f_\H(n,\ell,\lam)\,\wh g_\H(\ell,m,\lam).
\end{aligned}$$
Then, remembering Relation\refeq{FourierConvol} completes the proof of\refeq{newFourierconvoleq1}.

\end{document}